\documentclass[12pt]{amsart}
\usepackage{amsfonts,amsmath,amssymb,amsthm, geometry}
\usepackage[T1]{fontenc}
\usepackage[latin9]{inputenc}
\usepackage[all]{xy}
\usepackage[active]{srcltx}
\usepackage{graphicx}

\makeatletter

\newtheorem{theorem}{Theorem}[section]
\newtheorem{corollary}[theorem]{Corollary}
\newtheorem{lemma}[theorem]{Lemma}
\newtheorem{proposition}[theorem]{Proposition}
\theoremstyle{definition}
\newtheorem{definition}[theorem]{Definition}
\theoremstyle{remark}
\newtheorem{remark}[theorem]{\sc Remark}
\newtheorem{example}[theorem]{\sc Example}

\newenvironment{lyxlist}[1]
{\begin{list}{}
{\settowidth{\labelwidth}{#1}
 \setlength{\leftmargin}{\labelwidth}
 \addtolength{\leftmargin}{\labelsep}
 }}
{\end{list}}

\renewcommand{\int}{{\mathrm{int}}}

\newcommand{\Sing}{\mathrm{Sing\hspace{1pt}}}

\newcommand{\grad}{\mathop{\rm{grad}}\nolimits}
\newcommand{\ity}{{\infty}}

\renewcommand{\d}{{\mathrm{d}}}
\newcommand{\supp}{{\mathrm{supp}}}

\newcommand{\e}{\varepsilon}
\newcommand{\fin}{\hspace*{\fill}$\Box$}
\newcommand{\m}{\setminus}

\newcommand{\ord}{\mathrm{ord}}

\newcommand{\bR}{{\mathbb R}}
\newcommand{\bC}{{\mathbb C}}

\newcommand{\bP}{{\mathbb P}}
\newcommand{\bN}{{\mathbb N}}
\newcommand{\bZ}{{\mathbb Z}}

\newcommand{\bx}{{\mathbf{x}}}
\newcommand{\bz}{{\mathbf{z}}}

\begin{document}

\title[Bifurcation values of mixed polynomials]{Bifurcation values and monodromy of mixed polynomials}

\date{December 14, 2011}

\author{Ying Chen}
\author{Mihai Tib\u ar}
\address{Math\'ematiques, Laboratoire Paul Painlev\'e, Universit\'e Lille 1,
59655 Villeneuve d'Ascq, France.}
\email{tibar@math.univ-lille1.fr}
\email{Ying.Chen@math.univ-lille1.fr}



\thanks{The first named author acknowledges the support of the French \textit{Agence Nationale de la Recherche}, grant ANR-08-JCJC-0118-01}

\begin{abstract}
We study  the bifurcation values of real polynomial maps $f: \bR^{2n} \to \bR^2$ which reflect the lack of asymptotic  regularity at infinity. We formulate real counterparts of some structure results which have been previously proved in case of complex polynomials by Kushnirenko, N\'emethi and Zaharia and other authors, emphasizing the typical real phenomena which occur.
\end{abstract}
\maketitle


\section{Introduction}\label{s:intro}

For a complex polynomial function $f:\mathbb{C}^{n}\rightarrow\mathbb{C}$, it is well known that there is a $C^\ity$ locally trivial fibration $f_| : \bC^n \m f^{-1}(\Lambda) \to \bC \m \Lambda$ 
over the complement of some finite subset $\Lambda \subset \bC$, see e.g. \cite{Va}, \cite{Ve}. The minimal such $\Lambda$ is called the set of \emph{bifurcation values}, or the set of \emph{atypical values},  and shall be
denoted by $B(f)$. It was studied in several papers, such as \cite{Br1},
\cite{Br2}, \cite{Ne1}, \cite{NZ1}, \cite{ST}, \cite{Pa-compo} etc. Besides the critical values of $f$,  $B(f)$ may contain other values due to the asymptotical ``bad'' behaviour at infinity.

If one keeps only the real algebraic structure and views $f$ as a map $f: \bR^{2n} \to \bR^2$, it is natural to ask what can be still proved. We study the bifurcation locus of such real polynomial maps by regarding them as real maps $\bC^n \to \bC$, called ``mixed polynomials''  by Mutsuo Oka, who has studied in a recent series of papers \cite{Oka1, Oka2, Oka3} the topology of germs of mixed polynomials and mixed hypersurfaces.

For a polynomial map $F : \bR^m \to \bR^p$, $m>p$, the \textit{bifurcation locus} $B(F)$  is the minimal set such that $F$ is a $C^\ity$  locally trivial fibration over $\bR^p \m B(F)$. For $m=2$ and $p=1$ there exists a characterisation of $B(F)$, cf  \cite{TZ}, which is more involved than the one of the corresponding complex setting, cf \cite{HL}. For higher $p>1$,  by using a certain regularity condition at infinity, Kurdyka, Orro and Simon \cite{KOS} found a closed semi-algebraic set $K(F)$ including $B(F)$ and called it the set of \textit{generalised critical values}. 
In this paper we work with the $\rho$-regularity at infinity, a condition derived from Milnor's local condition. It allows us to exhibit a certain semi-algebraic closed set $S(f)$ of asymptotic ``bad'' values which estimates from above the set of atypical values at infinity and is related to the set $K_\infty(f)$ introduced by Kurdyka, Orro and Simon \cite{KOS}. Namely, we show by Proposition \ref{p:S_contained_in_K} and Fibration Theorem \ref{t:fib} that one has the inclusions $B(f) \subset S(f)\cup f(\Sing f) \subset K(f)$. The second one appears to be strict in general; we indicate  an example where the inclusion $S(f)\subset K_\infty(f)$ is strict (Remark \ref{r:compare_reg}).

In order to get a more effective estimation of $S(f)$, we focus to the class of Newton non-degenerate mixed polynomials, as defined by Oka \cite{Oka2}. In case of a Newton non-degenerate holomorphic polynomial $f:\bC^n \to \bC$ with $f(0) =0$, N\'emethi and Zaharia \cite[Theorem 2]{NZ1} defined the set $\mathfrak{B}$ of ``bad faces'' of the support $\supp (f)$ and showed the inclusion:
\begin{equation}\label{eq:nemethizaharia}
 B(f)\subset f(\Sing f)\cup \{0\}\cup\underset{\Delta\in\mathfrak{B}}{\cup}f_{\Delta}
(\Sing f_{\Delta}\cap \bC^{*n}). 
\end{equation}

In the mixed setting, the bifurcation set turns out to be of real dimension $\le 1$. After re-defining ``bad faces'', see Definition \ref{d:badface} and Remark \ref{r:bad}, we get a similarly looking estimation of $S(f)$:

\begin{theorem}\label{t:atyp_Newton}
 Let $f: \bR^{2n} \to \bR^{2}$ be a mixed polynomial which depends effectively on all the variables and let $f(0)=0$. If $f$ is Newton non-degenerate then:
\begin{enumerate}
\item $S(f)\subset\{0\} \cup \underset{\Delta\in\mathfrak{B}} \bigcup
f_{\Delta}(\Sing f_{\Delta}\cap \bC^{*n})$.
\item  If  $f$ is moreover Newton strongly non-degenerate then $f(\Sing f)$ and $S(f)$ are bounded.
\end{enumerate}

\end{theorem}
We present in \S \ref{ss:examples} several examples, some of them illustrating the differences to the holomorphic setting. In particular,  Example \ref{e:not_strongly} shows a phenomenon which could not occur before (i.e. for holomorphic polynomials):  a bad face which is in the same time a non-degenerate face of the Newton boundary, but which contributes to the bifurcation locus.  

It was proved in \cite[Proposition 6]{NZ1} that the inclusion Theorem \ref{t:atyp_Newton}(a) becomes an equality whenever $n=2$.  Example \ref{e:counterex} shows that this is no more the case in the mixed setting. In the same example we compute explicitly the bifurcation set $B(f)$ (which appears to be a real one-dimensional closed curve) and in particular the change of the topology of the fibres when crossing it. This change of topology and in particular the explicit description of $B(f)$ are given in Example \ref{ex:1} too.

 The proof of Theorem \ref{t:atyp_Newton} in \S \ref{s:proof} provides in addition some more precise information on how to detect the values $c\not=0$ such that the fibre $f^{-1}(c)$ contains unbounded branches of $M(f)$, see Remark \ref{r:cnotzero}. We also draw some consequences of Theorem \ref{t:atyp_Newton} for the convenient and the weighted-homogeneous mixed functions.

One may remark that Theorem \ref{t:atyp_Newton}(b) implies the existence of a \textit{monodromy fibration at infinity} which is quite unexpected in the real setting. This includes the special property that the image of $f$ contains the complement of some disk in $\bR^2$.
 We then prove the stability of the monodromy at infinity within a family of mixed polynomials with constant Newton boundary at infinity. This represents not only an extension of the corresponding result in the holomorphic case \cite[Theorem 17]{NZ2}, \cite[Theorem 1.1]{Ph}, but yields the following result, Corollary \ref{c:newton_constant}:
\textit{if $f$ and $g$ are two Newton strongly non-degenerate mixed polynomials with the same Newton boundary at infinity $\Gamma^+$ and such that the restrictions $f_{\Gamma^+}$ and $g_{\Gamma^+}$ are holomorphic, then their monodromies at infinity are isotopic.} This holds in spite of the fact that in the mixed setting, unlike in the complex setting,  the set of Newton (strongly) non-degenerate polynomials is neither connected nor dense, see \S \ref{ss:open}.


\section{Atypical values of mixed polynomials}\label{atyp}

In the setting of holomorphic polynomials, it is well-known that in certain cases $f$ has no atypical values at infinity, for instance:  ``convenient polynomials with non-degenerate
Newton principal part at infinity'' (Kushnirenko \cite{Ku}), see \S \ref{ss:newton}, polynomials which are ``tame'' (Broughton \cite{Br1}, \cite{Br2}), ``M-tame'' (N\'emethi 
\cite{Ne1}, \cite{Ne2}), ``cohomologically tame'' (Sabbah, N\'emethi \cite{NS}, \cite{Sa}). 
 For $n=2$ one has several characterisations of the atypical values at infinity, see e.g. \cite{HL}, \cite{Du}, \cite{Ti-reg}. In higher dimensions the problem is still open and one looks for some significant set $A \supset B(f)$  which bounds $B(f)$ reasonably well. For instance, in
 case of non-convenient but still Newton non-degenerate polynomials, N\'emethi and Zaharia \cite{NZ1} found an interesting approximation $A \supset B(f)$ in terms of certain faces of the support of $f$, see below. This provides a large class of polynomials for which we control rather well the bifurcation locus.

Let us now leave the holomorphic setting. 
We first set some notations and definitions, then show that the fibres of a mixed polynomial $f$ which are asymptotically tangent to the spheres may cause atypical behaviour at infinity and that the \textit{$\rho$-regularity} is weaker than some other regularity conditions at infinity. 


If $f:= (g,h) : \bR^{2n} \to \bR^2$, where $g(x_1, \ldots , y_n)$
 and $h(x_1, \ldots , y_n)$ are real polynomial functions, then, by writing $\mathbf{z}=\mathbf{x}+i\mathbf{y}\in \bC^n$, where $z_{k}=x_{k}+iy_{k}$ for $k=1,\ldots,n$, 
one gets a polynomial function $f: \bC^n \to \bC$ in variables $\bz$ and $\bar \bz$, namely 
$f(\mathbf{z},\mathbf{\bar{z}}):=g(\frac{\mathbf{z}+\mathbf{\bar{z}}}{2},\frac{\mathbf{z}-\mathbf{\bar{z}}}{2i})+ih(\frac{\mathbf{z}+\mathbf{\bar{z}}}{2},\frac{\mathbf{z}-\mathbf{\bar{z}}}{2i})$, and reciprocally.
 The notation $f(\mathbf{z},\mathbf{\bar{z}})$ instead of simply $f(\mathbf{z})$ is useful since we shall often use derivation with respect to $\mathbf{z}$ and $\mathbf{\bar{z}}$, such as in the following notations: $\d f : =\left(\frac{\partial f}{\partial z_{1}},\cdots,\frac{\partial f}{\partial z_{n}}\right)$,
$\overline{\d}f :=\left(\frac{\partial f}{\partial\overline{z}_{1}},\cdots\frac{\partial f}{\partial\overline{z}_{n}}\right)$, and 
 $\overline{\d f} := \left(\frac{\partial \bar f}{\partial\overline{z}_{1}},\cdots\frac{\partial \bar f}{\partial\overline{z}_{n}}\right)$ is the conjugate of $\d f$.
We shall therefore write $f$ as a \textit{mixed polynomial}\footnote{The concept appears in the work by A'Campo \cite{AC} and by Ruas, Seade and Verjovsky \cite{RSV} who actually studied a subclass of such mixed polynomials. One may lookup  \cite{Oka1, Oka2, Oka3} for more references.}, after Oka \cite{Oka2}:
\[
f(\mathbf{z}) = f(\mathbf{z},\mathbf{\bar{z}})=\underset{{\scriptstyle \nu,\mu}}{\sum}c_{\nu,\mu}\mathbf{z}^{\nu}\mathbf{\bar{z}^{\mu}}
\]
 where $c_{\nu,\mu}\in \bC$,  $\mathbf{z}^{\nu} :=z_{1}^{\nu_{1}}\cdots z_{n}^{\nu_{n}}$ for
$\nu=(\nu_{1},\cdots,\nu_{n})\in\mathbb{N}^{n}$ and  
 $\mathbf{\bar{z}^{\mu}} :=\bar{z}_{1}^{\mu_{1}}\cdots\bar{z}_{n}^{\mu_{n}}$
for $\mu=(\mu_{1},\cdots\mu_{n})\in\mathbb{N}^{n}$. 
 
 \begin{lemma}\label{l:M(f)geometric}
Let $f: \bC^n \to \bC$ be a mixed polynomial. The intersection of the fibre  $f^{-1}(f(\mathbf{z},\overline{\mathbf{z}}))$ with the sphere $S_r^{2n-1}$ of radius $r=\|\bz\|$ is not transversal at $\bz\in \bC^n\m \{ 0\}$
if and only if there exist $\mu\in\bC^*$, $\lambda\in\mathbb{R}$ such that: 
\[
\lambda\mathbf{z}=\mu\overline{\mathrm{d}f}(\mathbf{z},\overline{\mathbf{z}})+\overline{\mu}\overline{\mathrm{d}}f(\mathbf{z},\overline{\mathbf{z}}).\]
 \end{lemma}
 
\begin{proof}
Let $
f: \bC^n = \mathbb{R}^{2n}\longrightarrow\mathbb{R}^{2},\, f(\mathbf{z},\overline{\mathbf{z}}) =(\mathrm{Re}f(\mathbf{z},\overline{\mathbf{z}}),\mathrm{Im}f(\mathbf{z},\overline{\mathbf{z}}))$, 
  and let us denote $\mathbf{v}:=(x_{1},y_{1},\ldots,x_{n},y_{n})$.

If $f^{-1}(f(\mathbf{z},\overline{\mathbf{z}}))$ does not intersect
transversely the sphere $S_{r}^{2n-1}$ at \textbf{$\mathbf{z}$},
then there exist $\alpha,\beta,\gamma\in\mathbb{R},\left|\alpha\right|+\left|\beta\right|+\left|\gamma\right|\neq0$
such that:
\begin{equation*}
\gamma\mathbf{v}=\alpha\mathrm{dRe}f(\mathbf{v})+\beta\mathrm{dIm}f(\mathbf{v}).
\end{equation*}

By displaying this equality we easily get  $\gamma z_{k}=(\alpha+\beta i)\frac{\partial\bar{f}}{\partial\overline{z}_{k}}+(\alpha-\beta i)\frac{\partial f}{\partial\overline{z}_{k}}$
for every $k\in\{1,\ldots,n\}$.  Our claim follows by taking $\lambda=\gamma$ and $\mu=\alpha+\beta i$.
\end{proof}

The \emph{singular locus} $\Sing f$ of a mixed polynomial $f$ is by definition the set of critical
points of $f$ as a real-valued map. From Lemma \ref{l:M(f)geometric}, by taking 
$\lambda =0$ and dividing by $\mu$, we get the following characterisation:
\begin{lemma}\label{pro:-is-a}\cite[Proposition 1]{Oka1} \ 
One has $\mathbf{z}\in\Sing f$ if and only if there exists $\mu\in\mathbb{C}$,
$\left|\mu\right|=1$, such that $\overline{\mathrm{d}f}(\mathbf{z},\overline{\mathbf{z}})=\mu\overline{\mathrm{d}}f(\mathbf{z},\overline{\mathbf{z}})$. \fin
\end{lemma}

\subsection{$\rho$-regularity}

\begin{definition}\label{d:milnorset}
The \emph{Milnor set} of a mixed polynomial $f$ is 
\[  M(f)=\left\{ \mathbf{z}\in\mathbb{C}^{n}\mid\exists\lambda\in\mathbb{R} \mbox{ and } \mu\in\mathbb{C}^{*}, \mbox{ such that }\lambda\mathbf{z}=\mu\overline{\mathrm{d}f}(\mathbf{z},\overline{\mathbf{z}})+\overline{\mu}\overline{\mathrm{d}}f(\mathbf{z},\overline{\mathbf{z}})\right\} .\]
\end{definition}
In case of holomorphic $f$, $M(f)$ was called ``Milnor set'' in \cite{NZ1}. By its definition and by its geometric interpretation, $M(f)$ is a closed semi-algebraic subset of $\mathbb{C}^{n}$ and this fact will be used in the following. Lemma \ref{l:M(f)geometric} gives the geometric interpretation of $M(f)$ as the critical locus of the map $(f, \rho)$, where $\rho : \bR^{2n} \to \bR_{\ge 0}$ is the Euclidean distance function.
Like in the holomorphic setting \cite{NZ1}, one may define:

\begin{definition}\label{d:s}
The set of \emph{asymptotic} \emph{$\rho$-nonregular values} of
a mixed polynomial $f$ is \begin{align*}
S(f) & =\{c\in\mathbb{C\mid}\exists\,\{\mathbf{z}_{k}\}_{k\in\mathbb{N}}\subset M(f),\,\underset{k\rightarrow\infty}{\lim}\Vert\mathbf{z}_{k}\Vert=\infty\,\mathrm{and}\,\underset{k\rightarrow\infty}{\lim}\mathit{f(\mathbf{z}_{k},\overline{\mathbf{z}_{k}})=c}\}.
\end{align*}
\end{definition}

A value $c\not\in S(f)$ will be called an \emph{asymptotic $\rho$-regular value}\footnote{The name ``$\rho$-regularity'' was used in the setting of real and complex polynomial functions in \cite{Ti-reg, Ti}.}. The condition $S(f) = \emptyset$ for a holomorphic $f$ was called ``M-tameness'' in \cite{NZ1}, \cite{NZ2}, \cite{NS}.  

In order to investigate the properties of $S(f)$ we need a version of the \emph{curve selection lemma} at infinity. Milnor \cite{Mi} has proved this lemma 
at points of the closure of a semi-analytic set. N\'emethi and Zaharia  \cite{NZ1}, \cite{NZ2}, showed how to extend the result at infinity at some fibre of a holomorphic polynomial function. We give here a more general statement including the case when the value of $|f|$ tends to infinity. Since the proof is similar to the one in \cite{NZ2} and uses Milnor's result, we may safely leave it to the reader.
\begin{lemma}\label{l:curve-selection} \textbf{Curve Selection Lemma at infinity}\\
Let $U\subseteq\mathbb{R}^{n}$
be a semi-analytic set. Let $g:\mathbb{R}^{n}\longrightarrow\mathbb{R}$ be a polynomial function.
If there is $\left\{ \mathbf{y}_{k}\right\} _{k\in\mathbb{N}}\subset U$
such that $\underset{k\rightarrow\infty}{\lim}\Vert\mathbf{y}_{k}\Vert=\infty$
and $\underset{k\rightarrow\infty}{\lim}g(\mathbf{y}_{k})=c$, where
$c\in\mathbb{R}$, $c=\infty$ or $c=-\infty$, then there exist a
real analytic path $\mathbf{x}(t)=\mathbf{x_{0}}t^{\alpha}+\mathbf{x_{1}}t^{\alpha+1}+\mathrm{h.o.t.}$ defined on some small enough interval
$\left]0,\varepsilon\right[$ with $\mathbf{x}(t)\in U$,
such that $\mathbf{x_{0}}\neq 0$, $\alpha<0$, $\alpha\in\mathbb{Z}$, and
 that $\lim_{t \to 0} g(\mathbf{x}(t))=c$. 
 \fin
 \end{lemma}
We have the following structure result:
\begin{proposition}\label{p:closed}
\label{pro:-is-a-1} If  $f:\mathbb{C}^{n}\rightarrow\mathbb{C}$
is a mixed polynomial, then 
$S(f)$ and $f(\Sing f)\cup S(f)$ are closed semi-algebraic sets. 
\end{proposition}
\begin{proof}
 $S(f)$ may be presented as the projection of a semi-algebraic set. Indeed, consider the embedding of $\mathbb{\mathbb{C}}^{n}$
into $\mathbb{\mathbb{C}}^{n+1}\times\mathbb{\mathbb{C}}$ given by
the  semi-algebraic map:
\[
\varphi:\,(z_{1},\ldots,z_{n})\mapsto(\frac{z_{1}}{\sqrt{1+\| \mathbf{z}\|  ^{2}}},\ldots,\frac{z_{n}}{\sqrt{1+\| \mathbf{z}\|  ^{2}}},\frac{1}{\sqrt{1+\| \mathbf{z}\|  ^{2}}},f(\mathbf{z},\overline{\mathbf{z}})).
\]
Then $U_{1}:=\overline{\varphi(M(f))}\cap\{(x_{1},\ldots,x_{n+1},c)\in\mathbb{C}^{n+1}
\times\mathbb{C}\mid x_{n+1}=0\}$
is a semi-algebraic set and $S(f)=\pi(U_{1})$, where $\pi:\mathbb{C}^{n+1}\times\mathbb{C}\rightarrow\mathbb{C}$
 is the projection. Therefore $S(f)$ is semi-algebraic, by the Tarski-Seidenberg theorem. 

Let now $c\in \overline{S(f)}$. There exists a sequence $\{c_{i}\}_{i}\subset S(f)$ such that $\underset{i\rightarrow\infty}{\lim}c_{i}=c$.
For any $i$, we have by definition a sequence $\{\mathbf{z}_{i,n}\}_{n}\subset M(f)$
such that $\lim_{n\to \infty} \|\mathbf{z}_{i,n}\| =\infty$ and $\lim_{n\to \infty} f(\mathbf{z}_{i,n},\overline{\mathbf{z}}_{i,n})=c_{i}$.
Take a sequence $\{r_{i}\}_{i}\subset\mathbb{R}_{+}$  such that
$\lim_{i\to \infty}r_{i}=\infty$. For each $i$ there exists
$n(i)\in\mathbb{N}$ such that $\mathbf{z}_{i,n} >r_{i}$ implies $
\left| f(\mathbf{z}_{i,n},\overline{\mathbf{z}}_{i,n})-c_{i}\right| <\frac{1}{r_{i}}$,
$\forall n\geqslant n(i)$.
Setting $\mathbf{z}_{k} :=\mathbf{z}_{k, n(k)}$ we get a sequence $\{\mathbf{z}_{k}\}_{k}$ such that $\underset{k\rightarrow\infty}{\lim}\Vert\mathbf{z}_{k}\Vert=\infty$
and $\underset{k\rightarrow\infty}{\lim}f(\mathbf{z}_{k},\overline{\mathbf{z}}_{k})=c$, which shows that $c\in S(f)$.

Let now $a\in\overline{f(\Sing f)}\cup \overline{S(f)}$.
Since we have proved that $S(f)$ is closed, we may assume that $a\in\overline{f(\Sing f)}$.
Then there exists a sequence $\{\mathbf{z}_{n}\}_{j\in \bN}\subset\Sing f$,
such that $\lim_{j\rightarrow\infty}f(\mathbf{z}_{j},\overline{\mathbf{z}}_{j})=a$.
If $\{\mathbf{z}_{j}\}_{j\in \bN}$ is not bounded, then we may choose a subsequence
$\{\mathbf{z}_{j_k}\}_{k\in \bN}$ such that $\lim_{k\to\infty}\mathbf{z}_{j_k}=\infty$
and $\lim_{k\rightarrow\infty}f(\mathbf{z}_{j_k},\overline{\mathbf{z}}_{j_k})=a$.
Since $\Sing f\subset M(f)$, it follows that $a\in S(f)$, see also Remark \ref{r:cnotzero}. In the
other case, if $\{\mathbf{z}_{j}\}_{j\in \bN}$ is bounded, then we may choose a
subsequence $\{\mathbf{z}_{j_k}\}_{k\in \bN}$ such that $\lim_{k\to\infty}\mathbf{z}_{j_k}=\mathbf{z}_{0}$
and $\lim_{k\rightarrow\infty}f(\mathbf{z}_{j_k},\overline{\mathbf{z}}_{j_k})=a$.
Since $\Sing f$ is a closed algebraic set, this implies $\mathbf{z}_{0}\in\Sing f$, so $a=f(\mathbf{z}_{0},\overline{\mathbf{z}}_{0})\in f(\Sing f)$.
\end{proof}

\subsection{KOS-regularity}
For holomorphic polynomials one has the \textit{Malgrange regularity condition}, mentioned by F. Pham and used in many papers, see e.g. \cite{Pa-compo, ST}. This is known to be more general than ``tame'' (\cite{Br1, Br2}) or ``quasi-tame'' (\cite{Ne1, Ne2}).
It was extended to real maps by Kurdyka, Orro and Simon. These authors define in \cite{KOS} \textit{the set of generalized critical values}
$K(F)=F(\Sing F)\cup K_{\infty}(F)$ of a differentiable semi-algebraic map $F : \mathbb{\mathbb{R}}^{n} \to \bR^k$, where 
\begin{eqnarray*}
K_{\infty}(F)& := & \{y\in\bR^{k}\mid \exists \{ x_{l}\}_l \subset \bR^{n}, \| x_{l}\| \rightarrow\infty\\
 &  &  F(x_{l})\rightarrow y \mathrm{\ and\ } \| x_{l}\| \nu(\mathrm{d}F(x_{l}))\rightarrow0\}
\end{eqnarray*}
is the set of \emph{asymptotic critical values} of $F$. In this definition they use the following distance function:  
\begin{equation}\label{eq:dist}
\nu(A):=\underset{\| \varphi\| =1}{\inf}\| A^{*}\varphi\|  
\end{equation}
for $A\in L(\bR^n,\bR^k)$, where $A^{*}$ denotes its transpose. In the holomorphic setting one has $\nu(\mathrm{d}f(x)) = \| \grad f(x)\|$. The main result of \cite{KOS} is the following: 
\begin{theorem}\label{t:kos}\cite[Theorem 3.1]{KOS} \ \\
Let $F:\mathbb{R}^{n}\rightarrow\mathbb{R}^{k}$
be a differentiable semi-algebraic map. Then $K(F)$ is a closed semi-algebraic
set of dimension less than $k$. 

Moreover, if $F$ is of class $C^{2}$,
then $F:\mathbb{R}^{n}\setminus F^{-1}(K(F))\rightarrow\mathbb{R}^{k}\setminus K(F)$
is a locally trivial fibration over each connected component of $\mathbb{R}^{k}\setminus K(F)$.
In particular, the set $B(F)$ of bifurcation values of $F$ is included in $K(F)$.
\fin
\end{theorem}

\subsection{The fibration theorem}\label{ss:fibr}

By the next two results we prove that $S(f)$ contains the atypical values due to the asymptotical behaviour and that $S(f)$ is contained in $K_\ity(f)$. 

\begin{proposition}\label{p:S_contained_in_K}
Let $f : \bC^n \to \bC$ be a mixed polynomial. Then $S(f) \subset K_{\infty}(f)$.
\end{proposition}

\begin{remark}\label{r:compare_reg}
The above inclusion is \textit{strict} in general. This holds already in the holomorphic setting; to prove it, we may use the examples constructed by P\u aunescu and Zaharia in \cite{LZ}, as follows.  Let $f_{n,q}:\mathbb{C}^{3}\rightarrow\mathbb{C}$,
$f_{n,q}(x,y,z):=x-3x^{2n+1}y^{2q}+2x^{3n+1}y^{3q}+yz$, 
where $n,q\in\mathbb{N}\setminus\{0\}$.
These polynomials are $\rho$-regular at infinity and therefore we have
$S(f_{n.q})=\emptyset$. It was also shown in \cite{LZ} that $f_{n,q}$ satisfies Malgrange's
condition for any $t\in\mathbb{C}$ if and only if $n\leq q$.
Therefore, in case $n>q$, we have $\emptyset= S(f_{n.q})\subsetneq K_{\infty}(f_{n.q})\neq \emptyset$.
\end{remark}
\begin{proof}[Proof of Proposition \ref{p:S_contained_in_K}]
Let $(g,h)$ be the corresponding real map of the mixed polynomial $f$
and denote $\nu(\mathbf{x}) :=\nu(\d (g,h)(\mathbf{x}))$. Let us first show the equality:
\begin{equation}\label{lem:Let--be-2}
\nu(\mathbf{x})=\underset{\mu\in S^{1}}{\inf}\| \mu\overline{\mathrm{d}f}(\mathbf{z},\overline{\mathbf{z}})+\overline{\mu}\overline{\mathrm{d}}f(\mathbf{z},\overline{\mathbf{z}})\|  .
\end{equation}

By the definition (\ref{eq:dist}) of $\nu(\mathbf{x})$,
we have $\nu(\mathbf{x})=\underset{(a,b)\in S^{1}}{\inf}\| a\d g (\bx)+b\mathrm{d}h(\bx) \|$.
But the proof of Lemma \ref{l:M(f)geometric} shows the equality: $\|  a\mathrm{d}g(\bx) + b\mathrm{d}h(\bx) \|  = \| \mu\overline{\mathrm{d}f}(\mathbf{z},\overline{\mathbf{z}})+\overline{\mu}\overline{\d }f(\mathbf{z},\overline{\mathbf{z}})\|$
for $\mu=a+ib\in S^{1}$.  Our claim is proved.

Let then  $c\in S(f)$. By Definition \ref{d:s} and Lemma \ref{l:curve-selection}, there
exist real analytic paths, $\mathbf{z}(t)$ in  $M(f)$, $\lambda(t)$ in $\mathbb{R}$ and $\mu(t)$ in  $\mathbb{C}^{*}$,
defined on a small enough interval $\left]0,\varepsilon\right[$, 
such that $\lim_{t\to 0} \|\mathbf{z}(t)\| = \infty$ and $\lim_{t\to 0} f(\mathbf{z}(t),\overline{\mathbf{z}}(t))= c$ and 
that:
\begin{equation}\label{eq:3.1}
\lambda(t)\mathbf{z}(t)=\mu(t)\overline{\mathrm{d}f}(\mathbf{z}(t),\overline{\mathbf{z}}(t))+\overline{\mu}(t)\overline{\d }f(\mathbf{z}(t),\overline{\mathbf{z}}(t)).
\end{equation}
Let us assume that $\lambda(t)\not\equiv0$. Dividing \eqref{eq:3.1}
by $\| \mu(t)\|$ yields:
\begin{equation}\label{eq:3.3}
\lambda_{0}(t)\mathbf{z}(t)=\mu_{0}(t)\overline{\mathrm{d}f}(\mathbf{z}(t),\overline{\mathbf{z}}(t))+\overline{\mu}_{0}(t)\overline{\d }f(\mathbf{z}(t),\overline{\mathbf{z}}(t))
\end{equation}
where $\lambda_{0}(t):=\frac{\lambda(t)}{\|\mu(t)\|}$
and $\mu_{0}(t):=\frac{\mu(t)}{\|\mu(t)\|}$; therefore $\beta := \ord_{t}(\mu_{0}(t))= 0$.

Since $\underset{t\rightarrow0}{\lim}f(\mathbf{z}(t),\overline{\mathbf{z}}(t))=c$, we have
$\alpha:=\mathrm{ord_{t}}\frac{\d}{\d t} f(\mathbf{z}(t),\overline{\mathbf{z}}(t)) \ge 0$.
Then the following computation:
\begin{equation*}
\begin{array}{ll}
\overline{\mu}_{0}(t)\frac{\d}{\d t} f(\mathbf{z}(t),\overline{\mathbf{z}}(t))  +\mu_{0}(t)\frac{\d}{\d t} \overline{f}(\mathbf{z}(t),\overline{\mathbf{z}}(t)) & =  \left\langle \mu_{0}(t)\overline{\mathrm{d}f}(\mathbf{z}(t),\overline{\mathbf{z}}(t))+\overline{\mu}_{0}(t)\overline{\d }f(\mathbf{z}(t),\overline{\mathbf{z}}(t)),\frac{\d}{\d t} \mathbf{z}(t)\right\rangle \\
 &   +\left\langle \frac{\d}{\d t} \mathbf{z}(t),\mu_{0}(t)\overline{\mathrm{d}f}(\mathbf{z}(t),\overline{\mathbf{z}}(t))+\overline{\mu}_{0}(t)\overline{\d }f(\mathbf{z}(t),\overline{\mathbf{z}}(t))\right\rangle \\
 & \stackrel{\mathrm{by} (\ref{eq:3.3})}{=}  \lambda_{0}(t)(\mathrm{\left\langle \mathbf{z}(t),\frac{\d}{\d t} \mathbf{z}(t)\right\rangle +\left\langle \frac{\d}{\d t} \mathbf{z}(t),\mathbf{z}(t)\right\rangle })\\
 & =  \lambda_{0}(t)  \frac{\d}{\d t} \|  \mathbf{z}(t)\|^{2}
 \end{array}
\end{equation*}
shows that $\ord_{t}(\lambda_{0}(t)  \frac{\d}{\d t} \|\mathbf{z}(t)\|^{2})\geq\alpha+\beta\geq 0$. But since $\mathrm{ord_{t}}(\mathbf{z}(t))<0$, this implies that
 $\lim_{t\to 0}\left|\lambda_{0}(t)\right|\Vert\mathbf{z}(t)\Vert^{2}=0$. Note that this limit holds true for $\lambda(t)\equiv0$ too. 

From the last limit, by using \eqref{eq:3.3}, we get:
\begin{equation}\label{eq:3.4}
\underset{t\rightarrow0}{\lim}\|\mathbf{z}(t)\| \| \mu_{0}(t)\overline{\mathrm{d}f}(\mathbf{z}(t),\overline{\mathbf{z}}(t))+\overline{\mu}_{0}(t)\overline{\d }f(\mathbf{z}(t),\overline{\mathbf{z}}(t))\|  =0
\end{equation}
which, by \eqref{lem:Let--be-2}, implies
$\underset{t\rightarrow0}{\lim}\|\mathbf{x}(t)\| \|\nu(\mathbf{x}(t))\| =0$,
showing that $c\in K_{\infty}(f)$.
\end{proof}

\begin{theorem}\label{t:fib} \textbf{Fibration Theorem}\\
Let $f: \bC^n \to \bC$ be a mixed polynomial. Then the restriction:
\[ f_| : \bC^n \m f^{-1}(f(\Sing f)\cup S(f)) \to \bC \m f(\Sing f)\cup S(f) 
\]
is a locally trivial $C^\ity$ fibration over each connected component of $\bC \m (f(\Sing f)\cup S(f))$. 

In particular we have the inclusion $B(f) \subset f(\Sing f)\cup S(f)$.
\end{theorem}

\begin{remark}
 In the setting of mixed functions, our Theorem \ref{t:fib}
 extends \cite[Theorem 3.1]{KOS} since, by our Proposition \ref{p:S_contained_in_K} and Remark \ref{r:compare_reg} we have $S(f) \subsetneq K_\ity(f)$ and therefore we get a sharper approximation of the bifurcation set $B(f)$.  While our proof does not explicitly bound the dimension of $S(f)$, it follows from the preceding inclusion and from \cite[Theorem 3.1]{KOS}  that $S(f)$ has real dimension less than $2$.
\end{remark}

\begin{proof}[Proof of the Fibration Theorem]
Let $c\not\in f(\Sing f)\cup S(f)$. Then there is a closed disk $D$ centered at $c$ such that $D\subset \bC\m f(\Sing f)\cup S(f)$, since the latter is an open set by Proposition \ref{p:closed}. Let us first observe that there exists $R_0 \gg 0$ such that $M(f)\cap f^{-1}(D)\m  B_{R_0}^{2n} = \emptyset$.  Indeed, if this were not true, then there would exist a sequence $\{\mathbf{z}_{k}\}_{k\in\mathbb{N}}\subset f^{-1}(D)\cap M(f)$
such that $\underset{k\rightarrow\infty}{\lim}\Vert\mathbf{z}_{k}\Vert=\infty$.
Since $D$ is compact, there is a sub-sequence
$\{\mathbf{z}_{k_i}\}_{i\in\mathbb{N}}\subset M(f)$ and $c_0\in D$ such
that $\underset{i\rightarrow\infty}{\lim}\Vert\mathbf{z}_{k_i}\Vert=\infty\,\mathrm{and}\,\underset{i\rightarrow\infty}{\lim}f(\mathbf{z}_{k_i})=c_0$, 
which contradicts $D\subset\mathbb{C}\setminus S(f)$.

We claim that the map:
\begin{equation}\label{eq:triv2}
 f_| : f^{-1}(D)\m B_{R_0}^{2n} \to D
\end{equation}
is a trivial fibration on the manifold with boundary $(f^{-1}(D)\m B_{R_0}^{2n}, f^{-1}(D)\cap S_{R}^{2n-1})$, for any $R \ge R_0$. Indeed, this is a submersion by hypothesis but it is not proper, so one cannot apply Ehresmann's theorem directly. Instead, we consider the map $(f,\rho) : f^{-1}(D)\m B_{R_0}^{2n} \to D\times [R_0, \infty[$. As a direct consequence of its definition, this is a proper map. It is moreover a submersion since $\Sing (f,\rho) \cap f^{-1}(D)\m B_{R_0}^{2n} =\emptyset$ by the above remark concerning the set $M(f)$, which is nothing else but $\Sing (f,\rho)$. We then apply  Ehresmann's theorem to $(f,\rho)$ and conclude that it is a locally trivial fibration, hence a trivial fibration over $D\times [R_0, \infty[$. 
It follows that our map (\ref{eq:triv2}) is a trivial fibration too since it is the composition $\pi \circ (f,\rho)$, where $\pi : D\times [R_0, \infty[ \to D$ denotes the trivial projection.

Next observe that, since $D\cap f(\Sing f) = \emptyset$, the restriction:
\begin{equation}\label{eq:triv1}
 f_| : f^{-1}(D)\cap \bar B_{R_0}^{2n} \to D
\end{equation}
 is a proper submersion on the manifold with boundary $(f^{-1}(D)\cap \bar B_{R_0}^{2n}, f^{-1}(D)\cap S_{R_0}^{2n-1})$ and therefore a locally trivial fibration by Ehresmann's theorem, hence a trivial fibration over $D$.

We finally glue the two trivial fibrations \eqref{eq:triv1} and \eqref{eq:triv2} by the standard procedure.
\end{proof}

\section{Bifurcation values of Newton non-degenerate mixed polynomials}\label{s:newton}

We prove an estimation for the set of $\rho$-nonregular values at infinity under the condition of Newton non-degeneracy of the mixed polynomial.  We first introduce the necessary notions, then state the result. 
Let $\mathbb{C}^{*n} : = (\bC^*)^n$. 
\subsection{Newton boundary at infinity and non-degeneracy}\label{ss:newton}

Let $f$ be a mixed polynomial:

\begin{definition}\label{def:The-Newton-polyhedron}
We call $\mathrm{supp}\left(f\right)=\left\{ \nu+\mu\in\mathbb{N}^{n}\mid c_{\nu,\mu}\neq0\right\}$ 
the \textit{support} of $f$. We say that
$f$ is \emph{convenient} if the intersection of $\mathrm{supp}\left(f\right)$
with each coordinate axis is non-empty. We denote by $\overline{\mathrm{supp}(f)}$ the
convex hull of the set $\mathrm{supp}(f)\setminus\{0\}$. 
\noindent
The \emph{Newton polyhedron} of
a mixed polynomial $f$,
denoted by $\Gamma_{0}(f)$, is the convex hull of the set $\left\{ 0\right\} \cup\mathrm{supp}(f)$. The \emph{Newton boundary at infinity}, denoted by $\Gamma^+(f)$, is the union of the faces of the polyhedron
$\Gamma_{0}(f)$ which do not contain the origin. By ``face'' we
mean face of any dimension. 
\end{definition}

\begin{definition}\label{def:The-mixed-polynomial}
For any face $\Delta$ of $\overline{\mathrm{supp}(f)}$, we denote
the restriction of $f$ to $\Delta \cap \supp (f)$ by $f_{\Delta} :=  \sum_{\nu +\mu \in\Delta \cap \supp (f)} c_{\nu,\mu}\mathbf{z}^{\nu}\overline{\mathbf{z}}^{\mu}$.
The mixed polynomial $f$ is called
\emph{ non-degenerate} if
$\Sing f_{\Delta}\cap f_\Delta^{-1}(0) \cap \bC^{*n} = \emptyset$, for each face $\Delta$ of $\Gamma^+(f)$. Following Oka's terminology \cite{Oka2}, we say that $f$ is \textit{Newton strongly non-degenerate} if   $\Sing f_{\Delta} \cap \bC^{*n} = \emptyset$ for any face $\Delta$ of $\Gamma^+(f)$.

\end{definition}
The later condition is stronger and in general \textit{not equivalent} to the former, but they coincide in the holomorphic setting since $f_{\Delta}$ is quasi-homogeneous of non-zero degree.

  Kushnirenko \cite{Ku} had first introduced the Newton boundary of holomorphic \textit{germs}, which we denote by $\Gamma_-$ and which is different from $\Gamma^+$. Recently, Mutsuo Oka took over the program in the setting of mixed function germs and proved, among other results, the following \textit{local fibration theorem} :

\begin{theorem} \cite[Lemma 28, Theorem 29]{Oka2} \ \\
Let  $f : (\bC^n, 0) \to (\bC, 0)$ be the germ of a mixed polynomial which has a strongly non-degenerate and convenient  Newton boundary $\Gamma_-(f)$. Then $f$ has an isolated singularity at $0$ and the mapping:
\[ f_| : B_\e^{2n}\cap f^{-1}(D_\delta^*) \to D_\delta^* . \]
is a locally trivial fibration, for any small enough $\e > 0$ and $0 < \delta\ll \e$. \fin
\end{theorem}

In the setting of 
holomorphic polynomials, similar objects were studied by Broughton \cite{Br2}. He proved for instance that if $f$ is a complex polynomial with Newton non-degenerate and convenient polyhedron $\Gamma^+(f)$, then $S(f) = \emptyset$.  
  Later, N\'emethi and Zaharia  \cite{NZ1} dropped the convenience condition, defined the set $\mathfrak{B}$ of ``bad faces'' of $\overline{\supp f}$ and proved the result quoted in \S \ref{s:intro}.

\begin{remark}
If $f$ satisfies the conditions of Theorem \ref{t:atyp_Newton} except for $f(0)=0$, then we replace $f$ by $h=f-f(0)$ and apply to it Theorem \ref{t:atyp_Newton}. Since $\overline{\mathrm{d}f}(\mathbf{z},\overline{\mathbf{z}})=\overline{\mathrm{d}h}(\mathbf{z},\overline{\mathbf{z}})$
and $\overline{\d }f(\mathbf{z},\overline{\mathbf{z}})=\overline{\d }h(\mathbf{z},\overline{\mathbf{z}})$,
we get $M(f)=M(h)$ and $c\in S(f) \Leftrightarrow c-f(0)\in S(h)$. 
\end{remark}

Before giving the proof in \S \ref{s:proof}, we need to define the ingredients and prove several preliminary facts.

\subsection{The ``bad'' faces of the support}
 We consider a mixed polynomial $f : \bC^n \to \bC$, $f\not\equiv 0$.
\begin{definition}\label{d:badface}
 A face $\Delta\subseteq\overline{\mathrm{supp}(f)}$
is called \emph{bad} whenever: 

\begin{lyxlist}{00.00.0000}
\item [{(i)}] there exists a hyperplane $H\subset\mathbb{R}^{n}$ with
equation $a_{1}x_{1}+\cdots+a_{n}x_{n}=0$ (where $x_{1},\ldots,x_{n}$
are the coordinates of $\mathbb{R}^{n}$) such that: 

\item [{$\qquad\qquad$}] (a) there exist $i$ and $j$ with $a_{i}<0$
and $a_{j}>0$, \medskip{}
\item [{$\qquad\qquad$}] (b) $H\cap\overline{\mathrm{supp}(f)}=\Delta$.\end{lyxlist}

Let $\mathfrak{B}$ denote the set of bad faces of $\overline{\mathrm{supp}(f)}$.
 A face $\Delta \in\mathfrak{B}$ is called \textit{strictly bad} if it satisfies in addition the following condition:

\smallskip

\begin{lyxlist}{00.00.0000}
\item [{(ii)}] the affine subspace of the same dimension spanned by $\Delta$
contains the origin.
\end{lyxlist}

\end{definition}
\begin{remark}\label{r:bad}
In our Theorem \ref{t:atyp_Newton} we use the above definition for ``bad faces''.  For holomorphic mappings,   the set $\mathfrak{B}$ of bad faces used in the main formula \eqref{eq:nemethizaharia} of \cite{NZ1} corresponds to our definition of ``strictly bad faces''. 

Let us observe that not all bad faces are strictly bad. 
Nevertheless,  our Theorem \ref{t:atyp_Newton}(a) reduces in case of complex polynomials to precisely the statement \eqref{eq:nemethizaharia} of \cite{NZ1}.
 If $\Delta$ is a bad face which is not strictly bad, then it follows from the definitions that $\Delta$ is a face of $\Gamma^+(f)$. If we assume that $f$ is non-degenerate, then $\Delta$ is a non-degenerate face. If $f_\Delta$ is moreover holomorphic, then it follows that  $\Delta$ is strongly non-degenerate. Indeed,  there exists a hyperplane $V$ not passing through $0$ and such that $V\cap \overline{\mathrm{supp}(f)}=\Delta$, thus $f_\Delta$ is also weighted homogeneous of degree $\not= 0$ and therefore $\Sing f_\Delta \subset \{f_\Delta =0\}$. This shows in particular that in case of holomorphic $f$, the bad faces which are not strictly bad do not contribute with nonzero values in the formula of our Theorem \ref{t:atyp_Newton}(a), hence indeed only the strictly bad faces may play a role.
\end{remark}

The following lemma will be used in the proof of our theorem.

\begin{lemma}\label{lem:Let--be}
Let $l_{\mathbf{p}}(x)=\sum_{i=1}^{n}p_{i}x_{i}$
be a linear function such that $p=\underset{1\leq i\leq n}{\min}\{p_{i}\}<0$. We consider the
restriction of $l_{\mathbf{p}}(x)$ to $\overline{\mathrm{supp}(f)}$
and denote by $\Delta_\mathbf{p}$ the unique maximal face of $\overline{\mathrm{supp}(f)}$ (with respect to
the inclusion of faces) where $l_{\mathbf{p}}(x)$ takes its minimal
value $d_{\mathbf{p}}$. Let $d_{\mathbf{p}}\leq0$. 

\begin{lyxlist}{00.00.0000}
\rm \item [{(a)}] \it If $d_{\mathbf{p}}<0$, then $\Delta_\mathbf{p}$ is a face of $\Gamma^+(f)$.
\rm \item [{(b)}] \it If $d_{\mathbf{p}}=0$, then either $\Delta_\mathbf{p}$ is a face
of $\Gamma^+(f)$ or $\Delta_\mathbf{p}$ satisfies condition (ii) of Definition
\ref{d:badface}.
\end{lyxlist}
\end{lemma}
\begin{proof}
Let us first remark that from Definition \ref{def:The-Newton-polyhedron}
we have $\Gamma_{0}(f)=\mathrm{cone_{0}}(\Gamma^+(f))$, where $\mathrm{cone_{0}}(A)$
denotes the compact cone over the set $A$ with vertex the origin.
For each face $\Delta$ of $\Gamma_{0}(f)$ we have that either $\Delta$ is a face of $\Gamma^+(f)$ or $\Delta\ni 0$ and in this case we
have $\Delta =\mathrm{cone_{0}}(\Delta\cap\overline{\mathrm{supp}(f)})=\mathrm{cone_{0}}(\Delta \cap\Gamma^+(f))$.

Next, considering the restriction of $l_{\mathbf{p}}(x)$ to $\Gamma_{0}(f)$,
we denote by $\Delta_{1}$ the maximal face of $\Gamma_{0}(f)$
where $l_{\mathbf{p}}(x)$ takes its minimal value $d$. Note 
that $l_{\mathbf{p}}(x)$ can not attain its minimal value $d$ at
interior points of $\Gamma_{0}(f)$. Since $\Gamma^+(f)\subset\overline{\mathrm{supp}(f)}\subset\Gamma_{0}(f)$,
we have $d\leq d_{\mathbf{p}}$. 

(a). If $d_{\mathbf{p}}<0$ then  it follows by our initial remark that $\Delta_{1}$
is a face of $\Gamma^+(f)$, since otherwise we have $0\in\Delta_{1}$
and $d=0$. We therefore get $\Delta_\mathbf{p}=\Delta_{1}\subset\Gamma^+(f)$
and $d=d_{\mathbf{p}}$.

(b). If $d_{\mathbf{p}} =0$ and $\Delta_{1}$ is not a face of $\Gamma^+(f)$, then by the same initial remark we have $\Delta_{1}\ni0$ and therefore
$d=0$. Since $\Delta_{1}$ is the maximal face of $\Gamma_{0}(f)$
where $l_{\mathbf{p}}(x)$ takes its minimal value $d$, we get $\Delta_\mathbf{p}\subset\Delta_{1} \subset H$, where $H$ denotes the hyperplane $\{x\in\mathbb{R}^{n}\mid l_{\mathbf{p}}(x)=0\}$. 
We then have $\Delta_\mathbf{p}=\overline{\mathrm{supp}(f)}\cap H$, $\Delta_{1}=\Gamma_{0}(f)\cap H$,
and therefore $\Delta_\mathbf{p}=\Delta_{1}\cap\overline{\mathrm{supp}(f)}$.
Let us assume that $\Delta_\mathbf{p}$ does not verify condition (ii) of Definition
\ref{d:badface}, namely that we have $\dim\mathrm{cone_{0}}(\Delta_\mathbf{p})>\dim\Delta_\mathbf{p}$.
This implies that $\Delta_\mathbf{p}$ does not contain any interior point of $\mathrm{cone_{0}}(\Delta_\mathbf{p})$. By the initial remark,
$\Delta_{1}=\mathrm{cone_{0}}(\Delta_{1}\cap\Gamma^+(f))=
\mathrm{cone_{0}}(\Delta_\mathbf{p})$.
Then $\Delta_\mathbf{p}$ is a face of $\Gamma^+(f)$, which contradicts our assumption. 
\end{proof}

 Let $I \subset \{ 1, \ldots , n\}$. We shall use the following notations:

$\mathbb{C}^{I}=\left\{ (z_{1},\ldots,z_{n})\in \bC^n \mid z_{j}=0,j\notin I\right\} $,
and similarly $\mathbb{R}_{\ge 0}^{I}$,
 $\mathbb{C}^{*I} : =\mathbb{C}^{I} \cap \mathbb{C}^{*n}$,
$f^{I}:=f_{\mid\mathbb{C}^{I}}$.

From Definition \ref{def:The-Newton-polyhedron}, the faces of $f^{I}$
are among the faces of $f$, so we have the following: 

\begin{remark}\label{r:f_I}
Let $f$ be a mixed Newton (strongly) non-degenerate polynomial.
If $I\subset\left\{ 1,2,\ldots,n\right\}$ is such that $f^{I}$ is
not identically zero then:

(1) $f^{I}$ is a mixed Newton (strongly) non-degenerate polynomial.

(2) $\Gamma^+(f^{I})=\Gamma^+(f)\cap\mathbb{R}_{\ge 0}^{I}$.
\end{remark}

\noindent
We shall use the following fact for the restriction of $f$ to its
bad faces.

\begin{remark}\label{r:Delta_strongly}
 If a mixed polynomial $f$ is Newton (strongly) non-degenerate then, for any bad face $\Delta\subset\overline{\mathrm{supp}(f)}$,
$f_{\Delta}$ is Newton (strongly) non-degenerate. Indeed, any face $\Delta^{'}$ of $\Gamma^+(f_{\Delta})$ is also a subface
of $\Delta$, hence a subface of $\Gamma^+(f)$. The Newton (strong) 
non-degeneracy of $f$ implies that the restriction $f_{\Delta}$ is also Newton (strongly) 
non-degenerate.
\end{remark}

\subsection{Newton non-degeneracy is an open condition}\label{ss:open}

For a fixed polyhedron $\Gamma$ which is the Newton boundary at infinity of some mixed polynomial, we may define the subset 
$U^s_\Gamma :=\{[c_{1},c_{2},\ldots,c_{m}]\in\bP_\bC^{m-1}\mid \mbox{the polynomial } f_c(\mathbf{z},\overline{\mathbf{z}}) = f(\mathbf{z},\overline{\mathbf{z}}, c) =\sum_{j=1}^{m}c_{j}\mathbf{z}^{\mu_{j}}\overline{\mathbf{z}}^{\nu_{j}}$  is Newton strongly non-degenerate and  $\Gamma^+(f_c)= \Gamma \}$.  Similarly we define the set $U_\Gamma \supset U^s_\Gamma$ by just dropping the word ``strongly'' in the above definition. Then:

\begin{proposition}\label{p:nondeg}
 The subsets $U_\Gamma\subset \bP^{m-1}$ and $U^s_\Gamma\subset U_\Gamma$ of Newton non-degenerate and, respectively,  strongly non-degenerate mixed polynomials, with fixed Newton boundary $\Gamma$ at infinity, are semi-algebraic open sets.
\end{proposition}

\begin{remark} \label{r:strongly-nondeg}
 In the holomorphic setting one has ``Zariski-open'' instead of ``open'' and such a result was proved by Kushnirenko \cite{Ku} as a consequence of the Bertini-Sard theorem and of the the fact that ``strongly non-degenerate'' is equivalent to ``non-degenerate''.

Nevertheless in the real setting this proof does not apply and, in general, one does not have neither the connectedness, nor the density. Let us show by a simple example that Newton strong non-degeneracy does not insure  density.  Consider $f : \bC \to \bC$,  $f(z,\overline{z})=az^{2}+bz\overline{z}+c\overline{z}^{2}$,
where $a,b,c\in\mathbb{C}$. By direct computations using the homogeneity of $f$, we get that 
$f$ is Newton strongly non-degenerate if and only if $(\left|a\right|^{2}-\left|c\right|^{2})^{2}>\left|\overline{a}b-c\overline{b}\right|^{2}$. This inequality describes a homogeneous open set in $\bC^3$ which is not dense and not connected. Note also that $\supp (f)$ is a single point.
\end{remark}

\begin{proof}[Proof of Proposition \ref{p:nondeg}]
Let us show that $U^s_\Gamma$ is open and semi-algebraic.  The idea of this proof took its inspiration from Oka's alternate  proof in the holomorphic setting \cite[Appendix]{Oka3}.
For every face $\Delta\subset\Gamma$ we define:
\begin{align*}
V(\Delta) & :=\{(\mathbf{z},c)\in\mathbb{C}^{n}\times\mathbb{P}^{m-1} \mid 
\exists\lambda\in S^1_{1},\, \overline{\mathrm{d}f_{\Delta}}(\mathbf{z},\mathbf{\overline{\mathbf{z}}},c)=\lambda\overline{\mathrm{d}}f_{\Delta}(\mathbf{z},\overline{\mathbf{z}},c)\},\\
V(\Delta)^* & :=V(\Delta)\cap\{(\mathbf{z},c)\in\mathbb{C}^{n}\times\mathbb{P}^{m-1}\mid z_{1}z_{2}\ldots z_{n}\neq0\}.
\end{align*}

Note that $V(\Delta)$ is closed and that $\overline{V(\Delta)^*}=V(\Delta)$.
Let us consider the union $V^*=\cup_{\Delta\subset\Gamma}V(\Delta)^*$
and the projection $\pi:\mathbb{C}^{n}\times\mathbb{P}^{m-1}\rightarrow\mathbb{P}^{m-1}$.
Showing that $U^s_\Gamma$ is an open set means to prove that its complement $W=\pi(V^*)$ is a closed set. One  observes that $W$ is a semi-algebraic set,
since it is the projection of a semi-algebraic set. 

Let $c_{0}\in\overline{W}$.
By the Curve Selection Lemma, there exists a face $\Delta_0$ of $\Gamma$ and a real analytic path $(\mathbf{z}(t),c(t))\subset V(\Delta_0)^*$
defined on a small enough interval $\left]0,\varepsilon\right[$ such that $\lim_{t\to 0} c(t) = c_0$
and either $\lim_{t\to 0} \| \bz(t) \| = \ity$ or $\lim_{t\to 0} \bz(t)  = \bz_0 \in V(\Delta_{0})$.

Let then $z_{i}(t)= a_{i}t^{p_{i}} + \mathrm{h.o.t.}$ for $1\leq i\leq n$
where $a_{i}\neq0,p_{i}\in\mathbb{Z}$ and $\lambda(t)=\lambda_{0}+ \lambda_1 t + \mathrm{h.o.t.}$,
where $\lambda_{0}\in S^1_{1}$.  Let $\mathbf{a} :=(a_{1},\ldots,a_{n})\in\mathbb{C}^{*n}$,
$\mathbf{P}:=(p_{1},\ldots,p_{n})\in\mathbb{Z}^{n}$ and consider the
linear function $l_{\mathbf{P}}=\sum_{i=1}^{n}p_{i}x_{i}$ defined
on $\Delta_{0}$. Let $\Delta_{1}$ be the \emph{maximal face}
of $\Delta_{0}$ where $l_{\mathbf{P}}$ takes its minimal value,
say this value is $d_{\mathbf{P}}$.
We have: 
\[
\overline{\frac{\partial f_{\Delta_{1}}}{\partial z_{i}}}(\mathbf{a},\overline{\mathbf{a}}, c(t))t^{d_{\mathbf{p}}-p_{i}}+\mathrm{h.o.t.}=\lambda_{0}\frac{\partial f_{\Delta_{1}}}{\partial\overline{z}_{i}}(\mathbf{a},\overline{\mathbf{a}}, c(t))t^{d_{\mathbf{p}}-p_{i}}+\mathrm{h.o.t.}
\]
By taking the limit $c(t)\to c_0$ and focusing on the first terms of the expansions: 
\[
\overline{\mathrm{d}f_{\Delta_{1}}}(\mathbf{a},\overline{\mathbf{a}}, c_0)=\lambda_{0}\overline{\mathrm{d}}f_{\Delta_{1}}(\mathbf{a},\overline{\mathbf{a}}, c_0)
\]
we get that $(\mathbf{a},c_{0})\in V(\Delta_1)^* \subset V^*$, since $\mathbf{a}\in\mathbb{C}^{*n}$,
thus $c_{0}\in W$, which concludes the proof  that $W = \overline{W}$.

If in the definition of $V(\Delta)$ we add the supplementary equation $f_\Delta = 0$, then the same proof works for $U_\Gamma$ instead of  $U^s_\Gamma$. 
\end{proof}
%

\section{Proof of Theorem \ref{t:atyp_Newton}, some consequences and examples}\label{s:proof}

\subsection{Proof of Theorem \ref{t:atyp_Newton}(a).}

Let $c\in S(f)$. 
By Definition \ref{d:s} and Lemma \ref{l:curve-selection}, there
exist real analytic paths, $\mathbf{z}(t)$ in  $M(f)$, $\lambda(t)$ in $\mathbb{R}$ and $\mu(t)$ in  $\mathbb{C}^{*}$,
defined on a small enough interval $\left]0,\varepsilon\right[$, 
such that $\lim_{t\to 0} \|\mathbf{z}(t)\| = \infty$ and $\lim_{t\to 0} f(\mathbf{z}(t),\overline{\mathbf{z}}(t))= c$ and 
that:
\begin{equation}\label{eq:again}
\lambda(t)\mathbf{z}(t)=\mu(t)\overline{\mathrm{d}f}(\mathbf{z}(t),\overline{\mathbf{z}}(t))+\overline{\mu}(t)\overline{\d }f(\mathbf{z}(t),\overline{\mathbf{z}}(t)).
\end{equation}

Consider
the expansion of $f(\mathbf{z}(t),\overline{\mathbf{z}}(t))$. We have two situations, either:
\begin{equation}\label{eq:0}
f(\mathbf{z}(t),\overline{\mathbf{z}}(t)) \equiv c
\end{equation}
or
\begin{equation}\label{eq:-2}
f(\mathbf{z}(t),\overline{\mathbf{z}}(t))=c+bt^{\delta}+\mathrm{h.o.t.},\qquad\mathrm{where\ \ }  c, b\in\mathbb{C},\, b\neq0,\,\delta\in\mathbb{N}^{*}.
\end{equation}

\medskip

Let $I=\left\{ i\mid z_{i}(t)\not\not\equiv0\right\}$, observe that $I\neq\emptyset$
since $\underset{t\rightarrow0}{\lim}\Vert\mathbf{z}(t)\Vert=\infty$, and write: 
\begin{equation}\label{eq:-1}
z_{i}(t)=a_{i}t^{p_{i}}+\mathrm{h.o.t.,\qquad\mathit{\mathrm{where\ \ }a_{i}\neq}0}, \ p_{i}\in\mathbb{Z},\ i\in I.
\end{equation}
By eventually transposing the coordinates, we may assume that $I=\left\{ 1,\ldots,m\right\}$ and that $p= p_{1}\leq p_{2}\leq\cdots\leq p_{m}$.
Since $\underset{t\rightarrow0}{\lim}\Vert\mathbf{z}(t)\Vert=\infty$,
this implies $p =\underset{j\in I}{\min}\{p_{j}\}<0$. 
We denote $\mathbf{a}=(a_{1},\ldots,a_{m})\in\mathbb{C}^{*I}$, $\mathbf{p}=(p_{1},\ldots,p_{m})\in\bZ^{m}$
and consider the linear function $l_{\mathbf{p}}=\sum_{i=1}^{m}p_{i}x_{i}$
defined on $\overline{\supp(f^{I})}$.

Let us observe that since $f(0)=0$,  if $c\not= 0$, then $\overline{\supp(f^{I})}$ is not empty in both situations \eqref{eq:0} and \eqref{eq:-2}. 
 Let then $\Delta$
be the \emph{maximal face} of $\overline{\supp(f^{I})}$
where $l_{\mathbf{p}}$ takes its minimal value, say $d_{\mathbf{p}}$. We have:
\begin{equation}\label{eq:-3}
f(\mathbf{z}(t),\overline{\mathbf{z}}(t))=f^{I}(\mathbf{z}(t),\overline{\mathbf{z}}(t))=f_{\Delta}^{I}(\mathbf{a},\overline{\mathbf{a}})t^{d_{\mathbf{p}}}+\mathrm{h.o.t.}
\end{equation}
\noindent where $d_{\mathbf{p}}\leq\mathrm{ord_{t}}(f(\mathbf{z}(t),\overline{\mathbf{z}}(t)) = 0$. 

\noindent  In the following we keep the assumption\footnote{For the case
$c=0$, we refer to Remark \ref{rem:The-equality-}.}  $c\neq0$. 
 For $i\in I$ we have  the equalities: $\frac{\partial f}{\partial z_{i}}(\mathbf{z}(t),\overline{\mathbf{z}}(t))=\frac{\partial f^{I}}{\partial z_{i}}(\mathbf{z}(t),\overline{\mathbf{z}}(t))$
and $\frac{\partial f}{\partial\overline{z}_{i}}(\mathbf{z}(t),\overline{\mathbf{z}}(t))=\frac{\partial f^{I}}{\partial\overline{z}_{i}}(\mathbf{z}(t),\overline{\mathbf{z}}(t))$. Then we may write:
\begin{align}
\frac{\partial f}{\partial z_{i}}(\mathbf{z}(t),\overline{\mathbf{z}}(t)) & =\frac{\partial f_{\Delta}^{I}}{\partial z_{i}}(\mathbf{a},\overline{\mathbf{a}})t^{d_{\mathbf{p}}-p_{i}}+\mathrm{h.o.t.}\label{eq:-4}\\
\frac{\partial f}{\partial\overline{z}_{i}}(\mathbf{z}(t),\overline{\mathbf{z}}(t)) & =\frac{\partial f_{\Delta}^{I}}{\partial\overline{z}_{i}}(\mathbf{a},\overline{\mathbf{a}})t^{d_{\mathbf{p}}-p_{i}}+\mathrm{h.o.t.}\nonumber 
\end{align}

Consider the expansion of $\lambda(t)$, in case $\lambda(t)\not\equiv0$, and that of
$\mu(t)$:
\[
\lambda(t)=\lambda_{0}t^{\gamma}+\mathrm{h.o.t.,\qquad where\,\lambda_{0}\in\mathbb{R}}^{*},\,\gamma\in\mathbb{Z},\]
\[
\mu(t)=\mu_{0}t^{l}+\mathrm{h.o.t., \qquad where\,\mu_{0}\neq 0}, \ l\in\mathbb{Z}.
\]

Using all the expansions we get from \eqref{eq:again}, for any $i\in I$:
\[
(\mu_{0}\frac{\overline{\partial f_{\Delta}^{I}}}{\partial z_{i}}(\mathbf{a},\overline{\mathbf{a}})+\overline{\mu_{0}}\frac{\partial f_{\Delta}^{I}}{\partial\overline{z}_{i}}(\mathbf{a},\overline{\mathbf{a}}))t^{d_{\mathbf{p}}-p_{i}+l}+\mathrm{h.o.t.}=\lambda_{0}a_{i}t^{p_{i}+\gamma}+\mathrm{h.o.t.}
\]

Since $\lambda_{0}a_{i}\neq0$, comparing the orders of the two sides
in the above formula, we obtain:
\begin{equation}\label{eq:-5}
\mu_{0}\frac{\overline{\partial f_{\Delta}^{I}}}{\partial z_{i}}(\mathbf{a},\overline{\mathbf{a}})+\overline{\mu_{0}}\frac{\partial f_{\Delta}^{I}}{\partial\overline{z}_{i}}(\mathbf{a},\overline{\mathbf{a}})=\begin{cases}
\lambda_{0}a_{i}, & \quad\mathrm{if}\, d_{\mathbf{p}}-p_{i}+l=p_{i}+\gamma\\
\\0, & \quad\mathbf{\mathrm{if}}\, d_{\mathbf{p}}-p_{i}+l<p_{i}+\gamma\end{cases}
\end{equation}

Let $J=\left\{ j\in I\mid d_{\mathbf{p}}-p_{j}+l=p_{j}+\gamma\right\} $.
If we suppose that  $J\neq\emptyset$,
then $J=\{j\in I\mid p_{j}=p=\underset{j\in I}{\min}\{p_{j}\}<0\}$.In the situation \eqref{eq:-2} we have
$\frac{\mathrm{d}f(\mathbf{z}(t), \overline{\mathbf{z}}(t))}{\mathrm{d}t}=b\delta 
t^{\delta-1}+\mathrm{h.o.t}$ and on the other hand:
\begin{equation}\label{eq:compare} 
\begin{array}{cc}
{\displaystyle \frac{\mathrm{d}f(\mathbf{z}(t),\overline{\mathbf{z}}(t))}{\mathrm{d}t}} & ={\displaystyle \sum_{i=1}^{m}(\frac{\partial f}{\partial z_{i}}\cdot\frac{\partial z_{i}}{\partial t}+\frac{\partial f}{\partial\overline{z_{i}}}\cdot\frac{\partial\overline{z_{i}}}{\partial t})}={\displaystyle \sum_{i=1}^{m}(\frac{\partial f^{I}}{\partial z_{i}}\cdot\frac{\partial z_{i}}{\partial t}+\frac{\partial f^{I}}{\partial\overline{z_{i}}}\cdot\frac{\partial\overline{z_{i}}}{\partial t})}\\
 & =\left[\left\langle \mathbf{pa},\overline{\mathrm{d}f_{\Delta}^{I}}(\mathbf{a},\overline{\mathbf{a}})\right\rangle +\left\langle \mathbf{p}\mathbf{\overline{\mathbf{a}}},\overline{\bar{\d }f_{\Delta}^{I}(\mathbf{a},\overline{\mathbf{a}})}\right\rangle \right]t^{d_{\mathbf{p}}-1}+\mathrm{h.o.t.}
\end{array}
\end{equation}

\noindent
where $\mathbf{pa}=(p_{1}a_{1},\ldots,p_{m}a_{m})$. Comparing the
orders of the two expansions of $\frac{\mathrm{d}f(\mathbf{z}(t),\overline{\mathbf{z}}(t))}{\mathrm{d}t}$
and using the inequality $d_{\mathbf{p}}<\delta$ implied by $c\neq0$
(see after \eqref{eq:-3}), we find:

\begin{equation}\label{eq:-6}
\left\langle \mathbf{pa},\overline{\mathrm{d}f_{\Delta}^{I}}(\mathbf{a},\overline{\mathbf{a}})\right\rangle +\left\langle \mathbf{p}\mathbf{\overline{\mathbf{a}}},\overline{\bar{\d }f_{\Delta}^{I}(\mathbf{a},\overline{\mathbf{a}})}\right\rangle =0.
\end{equation}

Let us observe here that the proof of formula \eqref{eq:-6} holds
under the more general condition $d_{\mathbf{p}}<\delta$. 

Let now consider the situation \eqref{eq:0}. In this case the formula \eqref{eq:-6}
is true more directly,  since $\frac{\mathrm{d}f(\mathbf{z}(t), \overline{\mathbf{z}}(t))}{\mathrm{d}t}=0$ and after comparing this to \eqref{eq:compare}.

Next, multiplying \eqref{eq:-6} by $\overline{\mu}_{0}$ and taking the real part, we
get:
\begin{equation}\label{eq:innerprod}
\begin{array}{cc}
\mathrm{Re}\left\langle \mathbf{pa},\mu_{0}\overline{\mathrm{d}f_{\Delta}^{I}}(\mathbf{a},\overline{\mathbf{a}})\right\rangle +\mathrm{Re}\left\langle \mathbf{p}\mathbf{\overline{\mathbf{a}}},\mu_{0}\overline{\bar{\d }f_{\Delta}^{I}(\mathbf{a},\overline{\mathbf{a}})}\right\rangle \\
=\mathrm{Re}\left\langle \mathbf{p}\mathbf{a},\mu_{0}\overline{\mathrm{d}f_{\Delta}^{I}}(\mathbf{a},\overline{\mathbf{a}})+\overline{\mu}_{0}\bar{\d }f_{\Delta}^{I}(\mathbf{a},\overline{\mathbf{a}})\right\rangle =0.\end{array}
\end{equation}

On the other hand, from \eqref{eq:-5}, we have:
\[
\mathrm{Re}\left\langle \mathbf{p}\mathbf{a},\mu_{0}\overline{\mathrm{d}f_{\Delta}^{I}}(\mathbf{a},\overline{\mathbf{a}})+\overline{\mu}_{0}\bar{\d }f_{\Delta}^{I}(\mathbf{a},\overline{\mathbf{a}})\right\rangle =\sum_{i\in J}\lambda_{0}p\Vert a_{j}\Vert^{2}
\]
which is different from zero since $\lambda_{0}\neq0$, $p< 0$
and $a_{j}\neq0$. This contradicts formula \eqref{eq:innerprod}. We have therefore proved that
 $J=\emptyset$.

From \eqref{eq:-5} we obtain:
\begin{equation}\label{eq:-7}
\mu_{0}\overline{\mathrm{d}f_{\Delta}^{I}}(\mathbf{a},\overline{\mathbf{a}})+\overline{\mu}_{0}\bar{\d }f_{\Delta}^{I}(\mathbf{a},\overline{\mathbf{a}})=0.
\end{equation}

Let us observe that in case $\lambda(t)\equiv0$ we have $J=\emptyset$
and therefore we get directly \eqref{eq:-7}. 

What \eqref{eq:-7} tells us is 
that $\mathbf{a}$ is a singularity of $f_{\Delta}^{I}$.
 Set now $\mathbf{A=}(\mathbf{a},1,1,\ldots,1)$ with the
$i^{th}$ coordinate $z_{i}=1$ for $i\notin I$. Since $\Delta\subset\overline{\mathrm{supp(\mathit{f^{I}})}}$,
the restriction $f_{\Delta}$ does not depend on the variables
$z_{m+1},\ldots,z_{n}$ or their conjugates. Thus for any $i\in\{1,2,\ldots,n\}$, we
have $\frac{\partial f_{\Delta}}{\partial\overline{z}_{i}}(\mathbf{z}(t),\overline{\mathbf{z}}(t))=\frac{\partial f_{\Delta}^{I}}{\partial\overline{z}_{i}}(\mathbf{z}(t),\overline{\mathbf{z}}(t))$
and $\frac{\partial f_{\Delta}}{\partial z_{i}}(\mathbf{z}(t),\overline{\mathbf{z}}(t))=\frac{\partial f_{\Delta}^{I}}{\partial z_{i}}(\mathbf{z}(t),\overline{\mathbf{z}}(t))$.
By replacing $f_{\Delta}^{I}$ with $f_{\Delta}$ in \eqref{eq:-7},
we get that $\mathbf{A}\in\mathbb{C}^{*n}$ is a singularity of $f_{\Delta}$.

\noindent
We may now apply Lemma \ref{lem:Let--be} to $d_{\mathbf{p}}$ and $\Delta$.
We have the following two cases:

\medskip
\noindent
(I). If $d_{\mathbf{p}}<0$, then, by Lemma \ref{lem:Let--be}(a),
$\Delta$ is a face of $\Gamma^+(f^{I})$. Since
$\mathbf{A}\in\mathbb{C}^{*n}$ is a singularity of $f_{\Delta}$ and since we have $f_{\Delta}(\mathbf{A}, \overline{\mathbf{A}}) = 0$ by (\ref{eq:-3}) for $d_{\mathbf{p}}<0$, this contradicts the Newton non-degeneracy of $f$ (Definition \ref{def:The-mixed-polynomial})
assumed in the statement of Theorem \ref{t:atyp_Newton}.

\smallskip
\noindent
(II). Let $d_{\mathbf{p}}=0$. Then $c=f_{\Delta}^{I}(\mathbf{a},\overline{\mathbf{a}})=f_{\Delta}(\mathbf{A},\overline{\mathbf{A}})\in f_{\Delta}(\Sing f_{\Delta}\cap \bC^{*n})$. By Lemma \ref{lem:Let--be}(b), $\Delta$ is either a face of $\Gamma^+(f^{I})$ or satisfies the condition (ii) of Definition \ref{d:badface}. Note that these two conditions are exclusive, which fact follows immediately from the definitions.
Let us show that $\Delta$ is a bad face of $\overline{\mathrm{supp(\mathit{f})}}$.

Let $d$ denote the minimal value of the restriction of $l_{\mathbf{p}}$
to $\overline{\mathrm{supp}(f)}$. Since $\overline{\mathrm{supp(\mathit{f^{I}})}}=\overline{\mathrm{supp}(f)}\cap\mathbb{R}_{\ge 0}^{I}$,
we have $d\leq d_{\mathbf{p}}=0$. Let $H$ be the hyperplane defined by the
equation $\sum_{i=1}^{m}p_{i}x_{i}+q\sum_{i=m+1}^{n}x_{i}=0$, where$q>-d+1>0$. Then, for any $(x_{1},\ldots,x_{n})\in\overline{\mathrm{supp}(f)}
\setminus\overline{\mathrm{supp}(f^{I})}$,
the value of $\sum_{i=1}^{m}p_{i}x_{i}+q\sum_{i=m+1}^{n}x_{i}$ is
positive. We therefore get $\Delta=\overline{\mathrm{supp(\mathit{f^{I}})}}\cap H=\overline{\mathrm{supp}(f)}\cap H$.

If $\Delta$ does not satisfy condition (i)(a) of
Definition \ref{d:badface}, then we have $m=n$ and $p_{i}\leq0$
for all $1\leq i\leq n$. Since by hypothesis $f$ depends effectively on all variables, in particular on the variable $z_{1}$,  the value $d_{\mathbf{p}}$ must be negative, which is a contradiction to the above original assumption.

This ends our proof.\fin

\begin{remark}\label{rem:The-equality-}
The equality \eqref{eq:-7} is the key of
the above proof of Theorem \ref{t:atyp_Newton}(a). If $c=0$, then we have
two cases in situation \eqref{eq:-2}:

(1) If $d_{\mathbf{p}}=\ord_{t}(f(\mathbf{z}(t),\overline{\mathbf{z}}(t))$, then formula \eqref{eq:-7} might
be not true.

(2) If $d_{\mathbf{p}}<\ord_{t}(f(\mathbf{z}(t),\overline{\mathbf{z}}(t))$, then we get the same proof of formula
\eqref{eq:-7} as in Proof of (a) (see the remark after formula \eqref{eq:-6}).
\end{remark}
\begin{remark}\label{r:cnotzero}
Let 
 $\Sigma^\ity := \{ c\in \bC \mid f^{-1}(c) \cap M(f) \mbox{ is not bounded}\}$.
Under the hypotheses of Theorem \ref{t:atyp_Newton}, the above proof also shows that if $c\in \Sigma^\ity$ and $c\not= 0$ then $c$ is a critical value of $f_\Delta$, for some bad face $\Delta$.
Indeed, if the path $\bz(t)\subset M(f) \cap f^{-1}(c)$ is not bounded, then it must be included in the singular locus $\Sing f^{-1}(c)$  since the fibre $f^{-1}(c)$ is an algebraic set. (An alternate argument may be extracted from the last part of the proof of Proposition \ref{p:S_contained_in_K}). This shows the inclusion $\Sigma^\ity \subset S(f) \cap f(\Sing f)$. By Theorem \ref{t:atyp_Newton}(a) we then have $\Sigma^\ity \m \{ 0\} \subset \underset{\Delta\in\mathfrak{B}} \bigcup
f_{\Delta}(\Sing f_{\Delta})$.
\end{remark}


\subsection{Proof of Theorem \ref{t:atyp_Newton}(b).}


By absurd, let us suppose $f(\Sing f)$ is not bounded. Since $\Sing f$
is a semi-algebraic set, by Lemma \ref{l:curve-selection},
there exists  a real analytic path $\mathbf{z}(t)\subset\Sing f$
defined on a small enough interval $\left]0,\varepsilon\right[$ such
that: \[
\underset{t\rightarrow0}{\lim}\Vert\mathbf{z}(t)\Vert=\infty,\,\mathrm{and\,}\underset{t\rightarrow0}{\lim}\left|f(\mathbf{z}(t),\overline{\mathbf{z}}(t))\right|=\infty\]

We follow the proof of (a). Since $\mathbf{z}(t)\subset\Sing f$,  
we have $\lambda(t)\equiv0$ in \eqref{eq:again} and therefore we obtain \eqref{eq:-7} directly, as remarked after it. From $\underset{t\rightarrow0}{\lim}\left|f(\mathbf{z}(t),\overline{\mathbf{z}}(t))\right|=\infty$
it follows that $d_{\mathbf{p}}\leq\mathrm{ord_{t}}(f(\mathbf{z}(t),\overline{\mathbf{z}}(t))<0$.
We are in the situation of (I) from the proof of Theorem \ref{t:atyp_Newton}(a) but without being able to insure the equality $f_{\Delta}(\mathbf{A}, \overline{\mathbf{A}}) = 0$. That is why we  need here the Newton strong non-degeneracy in order to get a contradiction.

To prove that $f_{\Delta}(\Sing f_{\Delta})$ is bounded, for any bad face $\Delta\subset\overline{\mathrm{supp(\mathit{f})}}$, we use Remark \ref{r:Delta_strongly} and  the above proof for $f_\Delta$ in place of $f$.

Since $\overline{\mathrm{supp(\mathit{f})}}$ has finitely many faces and since, by Theorem \ref{t:atyp_Newton}(a), we have the inclusion $S(f)\subset\{0\}\cup\underset{\Delta\in\mathfrak{B}}{\cup}f_{\Delta}(\Sing f_{\Delta})$,
it follows that $S(f)$ is bounded. \fin



\bigskip
\subsection{Some consequences}
 We get some sharper statements for significant particular classes of non-degenerate mixed polynomials.  The following result extends the one for holomorphic polynomials proved in \cite{Ku}.
\begin{corollary}\label{c:convenient}
\noindent If $f$ is a mixed Newton non-degenerate and convenient
polynomial, then $S(f)=\emptyset$.
\end{corollary}
\begin{proof}
Under the same notations and definitions as in the proof of Theorem
\ref{t:atyp_Newton}(a), since $l_{\mathbf{p}}(x)=\sum_{i=1}^{m}p_{i}x_{i}$
has at least a coefficient $p_{j}<0$ for some $j$ and the intersection
of $\mathrm{supp}\left(f\right)$ with each positive coordinate axis
is non-empty, the value of $l_{\mathbf{p}}(x)$ at a point of the
intersection of $\mathrm{supp}\left(f\right)$ with the $j$-axis
is negative. This implies that the minimal value $d_{\mathbf{p}}$
is negative. By Lemma \ref{lem:Let--be}(a), $\Delta$ is a face
of $\Gamma^+(f)$.

Since we have here $d_{\mathbf{p}}<\ord_{t}(f(\mathbf{z}(t),\overline{\mathbf{z}}(t))$, by using Remark
\ref{rem:The-equality-}, we get formula \eqref{eq:-7} and 
a singularity $\mathbf{A}\in\bC^{*n}$ of $f_{\Delta}$ with $f_{\Delta}(\mathbf{A}) =0$ as in (I) above. This contradicts the Newton non-degeneracy of $f$. 
\end{proof}

\begin{definition} \label{d:weighted}
A mixed polynomial $f$ is called \emph{(radial)
weighted-homogeneous} if there exist positive integers $q_{1},\cdots,q_{n}$
with $\gcd(q_{1},\cdots,q_{n})=1$ and a positive integer $m$ such
that $\sum_{j=1}^{n}q_{j}(\nu_{j}+\mu_{j})=m$, or, equivalently, such that  $f(t\circ\mathbf{z})=t^{m}f(\mathbf{z},\overline{\mathbf{z}})$ for any $t\in\mathbb{R}^{*}$, where $t\circ\mathbf{z} :=(t^{q_{1}}z_{1},\ldots,t^{q_{n}}z_{n})$.
\end{definition}

\begin{corollary} 
Let $f$ be a mixed polynomial, weighted-homogeneous
 and Newton strongly non-degenerate. Then: 
\begin{enumerate}
\item $\Sing f\cap\mathbb{C}^{*n}=\emptyset$, 
\item $S(f)\cup f(\Sing f) \subset\left\{ 0\right\}$. 
\end{enumerate}
\end{corollary} 
\begin{proof}
Since $f$ is weighted-homogeneous, let's say of degree $m$, we have $f(0)=0$ and $\overline{\supp(f)}$
is contained in a single hyperplane which does not pass through the origin.
Therefore the Newton boundary $\Gamma^{+}(f)$ has a single maximal face and
its non-degeneracy implies $\Sing f\cap\mathbb{C}^{*n}=\emptyset$.
Since $\overline{\supp(f)}$ has no bad face and since by Theorem
\ref{t:atyp_Newton}(a) we have $S(f)\subset\{0\}\cup\underset{\Delta\in\mathfrak{B}}{\cup}f_{\Delta}(\Sing f_{\Delta})$,
it follows that $S(f)\subset\{0\}$. 

By absurd, let us suppose that $c\in f(\Sing f)\cap\mathbb{C}^{*}$.
For any $\mathbf{z}\in\Sing f$ such that $f(\mathbf{z},\overline{\mathbf{z}})=c$,
there exists $\lambda\in S_{1}^{1}$ such that $\overline{\mathrm{d}f}(\mathbf{z},\overline{\mathbf{z}})=\lambda\overline{\d}f(\mathbf{z},    \overline{\mathbf{z}})$.
Multiplying  by $t^{m-q_{i}}$ the equalities $\overline{\frac{\partial f}{\partial z_{i}}}(\mathbf{z},\overline{\mathbf{z}})=\lambda\frac{\partial f}{\partial\overline{z}_{i}}(\mathbf{z},\overline{\mathbf{z}})$
for $i=1,2\ldots,n$, and using that $f$ is weighted-homogeneous, we get that $\overline{\mathrm{d}f}(t\circ\mathbf{z},t\circ\overline{\mathbf{z}})=\lambda\overline{\d}f(t\circ\mathbf{z},t\circ\overline{\mathbf{z}})$.
 This implies that $t\circ\mathbf{z}\in\Sing f$ and $t^{m}c\in f(\Sing f)$, therefore
 $f(\Sing f)$ is not bounded, which contradicts Theorem \ref{t:atyp_Newton}(b). This proves that $f(\Sing f)\subset\{0\}$. 
\end{proof}

\subsection{Examples}\label{ss:examples}

\begin{example}\label{ex:2} 
The polynomial $f :\bC^2 \to \bC$, $f=z_{1}+z_{2}+\overline{z}_{1}^{2}+\overline{z}_{2}^{2}$, 
 is Newton strongly non-degenerate and convenient. By direct computation of $M(f)$ we obtain that $S(f)=\emptyset$, as predicted by Corollary \ref{c:convenient}, and  $f(\Sing f)= \{ a+\frac{1}{2}\overline{a}^{2} \mid a\in S^{1}\}$,  a closed cuspidal curve, which agrees with Theorem \ref{t:atyp_Newton}(b).
\end{example}

\begin{example}\label{e:counterex}
N\' emethi
and Zaharia have proved in \cite[Proposition 6]{NZ1}
that if the holomorphic polynomial $f:\mathbb{C}^{2}\rightarrow\mathbb{C}$
is Newton non-degenerate,  not convenient, not depending of just one variable and with $f(0) = 0$, then one has the equality $B(f)=f(\Sing f)\cup\{0\}\cup\underset{\Delta\in\mathfrak{B}}{\cup}f_{\Delta}(\Sing f_{\Delta})$, and in particular $0\in B(f)$.

Let us show that this is no more true  for mixed polynomials, by using the example $f : \bR^4 \to \bR^2$, 
$f(z_{1},z_{2})=z_{1}(1+\left|z_{2}\right|^{2}+z_{1}z_{2}^{4})$. This is Newton strongly non-degenerate,  not convenient and $f(0) = 0$. Standard computations yield that $f(\Sing f) = \emptyset$, hence $0\notin f(\Sing f)$. It is more tedious to show that $0\notin K_\ity(f)$ by using an argument based on the curve selection lemma \ref{l:curve-selection}. Then use Proposition \ref{p:S_contained_in_K} to get $0\notin S(f)$ and Theorem \ref{t:fib} to conclude $0\notin B(f)$.

Finally, let us compute explicitly $B(f)$. The above named computations also show the equalities $K_\ity(f) = S(f) = \{ c\in \bC \mid |c| = 1/4 \}$.  Let us then take some $c$ with $|c| > 1/4$. Using Theorem \ref{t:monodro_family} which will be proved in the next section, i.e., the stability of the monodromy at infinity in certain families,  we get the homeomorphism $f^{-1}(c) \simeq g^{-1}(c)$, where $g = z_1(1+ z_1z_2^4)$. By a direct computation we get the homotopy equivalence $g^{-1}(c)\simeq \vee_4 S^1$.  If one takes some $c$ with $|c| < 1/4$, then a similar computation shows that $f^{-1}(c)$ is homotopy equivalent to $\bC \sqcup \bC^*$. Together with Theorem \ref{t:fib}, this shows $B(f) = S(f) = \{ c\in \bC \mid |c| = 1/4 \}$.

\end{example}

\begin{example}\label{ex:1}
Let
 $f :\bR^4 \to \bR^2$, $f =z_{1}z_{2}+\overline{z}_{1}^{2}\overline{z}_{2}^{2}$.
This is a Newton strongly non-degenerate mixed polynomial, where $\Gamma^+(f) = (2,2)$ and $\overline{\supp (f)}$ consists of just one face $\Delta$ which is a bad face. 
The solution of $\overline{\mathrm{d}f}(\mathbf{z},\overline{\mathbf{z}})=\lambda\overline{\d}f(\mathbf{z},\overline{\mathbf{z}})$ for $\lambda\in S_1^1$, 
is the set $\{z_{1}z_{2}=\frac{1}{2\overline{\lambda}}\} \cup \{z_{1}=z_{2}=0\}$.
We obtain $f(\Sing f_\Delta)=f(\Sing f)=\{0\}\cup\{\frac{1}{2\overline{\lambda}}+\frac{1}{4\lambda^{2}}\mid\lambda\in S_{1}^{1}\}$.
By taking $z_{1}z_{2}=\frac{1}{2\overline{\lambda}}$ with $z_{1}\rightarrow0$, hence
 $z_{2}\rightarrow\infty$, by straightforward computations we get $f(\Sing f)\setminus\{0\}\subset S(f)$ and $\{0\} \not \in S(f)$.
On the other hand, for $\{\mathbf{z}^{k}\}_{k\in\mathbb{N}}\subset M(f)\setminus\Sing f$
such that $\underset{k\rightarrow\infty}{\lim}\Vert\mathbf{z}^{k}\Vert=\infty$, we get 
 $\left|f(\mathbf{z}^{k})\right|\rightarrow\infty$. This shows that $S(f) \m f(\Sing(f)) = \emptyset$, by using Theorem \ref{t:atyp_Newton}(b) too.
 Moreover, it shows
that the inclusion of Theorem \ref{t:atyp_Newton}(a) may be strict.

One may also compute explicitly the topology of the fibres with the method described in the preceding example. Let us note that the complement of $S(f)$, which is a simple closed plane curve containing the origin,  has two connected components.  The fibre of $f$ over some point of the exterior of this curve is homotopy equivalent to $\bC^* \sqcup \bC^*$ (by using a deformation from $f$  to $g := \overline{z}_{1}^{2}\overline{z}_{2}^{2}$ and then Theorem  \ref{t:monodro_family}).
One computes directly that a fibre over some interior point, different from the origin, is homotopy equivalent to the disjoint union of four $\bC^*$ and that the fibre over the origin is
 homotopy equivalent to the disjoint union of three $\bC^*$ and one point. In particular these computations, showing the change of topology of fibres, provide explicitly the bifurcation set  $B(f) = \{ 0\} \cup \{\frac{1}{2\overline{\lambda}}+\frac{1}{4\lambda^{2}}\mid\lambda\in S_{1}^{1}\}$.
\end{example}

\begin{example}\label{e:not_strongly}
The following is an example of a  Newton non-degenerate, not strongly non-degenerate, mixed polynomial. It also shows that bad faces which are not strictly bad may contribute to the bifurcation set $B(f)$, a phenomenon which does occur for  holomorphic functions (compare Theorem \ref{t:atyp_Newton} to Némethi-Zaharia statement \cite{NZ1}, cf \eqref{eq:nemethizaharia}). Let  $f:\mathbb{C}^{2}\rightarrow\mathbb{C}$,
$f=\left|z_{1}\right|^{2}(z_{2}^{2}+2z_{2}\overline{z}_{2}+1)$.

The support $\overline{\mathrm{supp}(f)}$ has three faces, all of which are contained in $\Gamma^{+}(f$), and the restrictions of $f$ look as follows: $f_{\triangle_{1}}=\left|z_{1}\right|^{2}(z_{2}^{2}+2z_{2}\overline{z}_{2})$, $f_{\triangle_{2}}=\left|z_{1}\right|^{2}$ and $f_{\triangle_{3}}=f$.
We observe that $\left\{ f_{\triangle_{i}}=0\right\} \cap\mathbb{C}^{*2}=\emptyset$, for $i=1,2,3$,  so $f$ is non-degenerate. However, $f_{\triangle_{2}}$ is not strongly non-degenerate.
 There is a single bad face $\triangle_{1}$ and it is not strictly bad.

By Theorem \ref{t:atyp_Newton}(a) we have  $S(f)\subset \left\{ 0\right\}\cup f_{\triangle_{1}}(\Sing f_{\triangle_{1}}\cap\mathbb{C}^{*2})$, and by straightforward computations we get $f(\Sing f)= \mathbb{R}_{\ge 0}$ and $f_{\triangle_{1}}(\Sing f_{\triangle_{1}}\cap\mathbb{C}^{*2})=\{ (\frac{3}{2}\pm\frac{\sqrt{3}}{2}i)t \mid t>0\}$.

Let us show that the inclusion  $B(f)\subset S(f)\cup f(\Sing f)$ of the Fibration Theorem \ref{t:fib} is an equality. Assuming that $f(\Sing f)\subset B(f)$ by definition, it remains to prove the inclusion $S(f)\subset B(f)$.
We fix some $t>0$ and consider a small disk neighborhood $U$ of the point $(\frac{3}{2}+\frac{\sqrt{3}}{2}i)t$. Let $\left|z_{1}\right|^{2}=c\neq0$, $z_{2}=x+iy$. The equality $f=a+ib$ yields: $c(3x^{2}+y^{2}+1) = a$, $-2cxy  =b$, and by combing these two equations we obtain:
\begin{equation}\label{eq:exmpl}
3bx^{2}+2axy+b(y^{2}+1)=0.
\end{equation}
Solving in $x$, the discriminant $4a^{2}y^{2}-12b^{2}(y^{2}+1)$ shows that
\eqref{eq:exmpl} has no solution if and
only if $a^{2}<3b^{2}$. This implies that the fibres of $f$ are empty over one half of the disk $U$ and non-empty over the other half, so there is no locally trivial fibration at $(\frac{3}{2}+\frac{\sqrt{3}}{2}i)t$. The same proof applies to  $(\frac{3}{2}-\frac{\sqrt{3}}{2}i)t$.
Altogether these yield the claimed inclusion $S(f)\subset B(f)$. 
\end{example}


\section{Families of mixed polynomials and stability of the monodromy at infinity}\label{s:monodromy}

As a consequence of Theorems \ref{t:fib} and \ref{t:atyp_Newton}(b), the class of Newton strongly non-degenerate polynomials $f$  has the property that $B(f) $ is bounded.  One has the following general definition.

\begin{definition}\label{d:monodromy_infty} \textbf{(Monodromy at infinity)}\\
Let $f : \bR^{2n} \to \bR^2$ be a real polynomial map and assume that the bifurcation set $B(f)$ is bounded. Let $\delta_0 >0$ such that $B(f)$ is included in the open disk $D_{\delta_0}$ of radius  $\delta_0$ centered at $0\in \bC$. We call \textit{monodromy (fibration) at infinity} the fibration:
 \[
f_{\mid}:f^{-1}(S_{\delta}^{1})\rightarrow S_{\delta}^{1}.\]
over some circle $S_{\delta}^{1}$ of radius $\delta$ which, by the Fibration Theorem \ref{t:fib}, exists and is independent of $\delta \ge \delta_0$.
\end{definition}

We then prove the following result:

\begin{theorem}\label{t:monodro_family}
Let $F_{s}(\mathbf{z},\overline{\mathbf{z}}):=F(\mathbf{z},\overline{\mathbf{z}},s) : \bR^{2n} \to \bR^2$
be a family of Newton strongly non-degenerate polynomials depending
analytically of a parameter $s$, where $s\in\left[0,1\right]$. If  the Newton boundary $\Gamma^{+}(F_{s})$
is constant in this family, then the monodromy at infinity is
stable\footnote{Here, ``stable'' means that the monodromy fibrations at infinity are equivalent in
this family.}.
\end{theorem}

 In the holomorphic setting such a result was proved first by N\' emethi and Zaharia in the convenient case  \cite[Theorem 17]{NZ2}, then extended  by Pham \cite{Ph} to non-convenient. In our mixed setting,  the technique developed for the proof of Theorem \ref{t:atyp_Newton}  enables us to pursue the extension of these results for families of mixed polynomials, along the pattern of \cite{NZ2} and \cite[Lemmas 3.2--3.5]{Ph}. 
Let us point out that in the holomorphic case the Newton non-degeneracy is
a Zariski open dense and connected condition, hence there exists
a family of Newton non-degenerate polynomials with the same Newton
boundary at infinity joining any two such polynomials. However, in the mixed case,
 we have shown in Remark \ref{r:strongly-nondeg} that the Newton strongly non-degenerate
condition is neither dense, nor connected, but it is still an open condition (see \S \ref{ss:open}). 
Therefore, in order to obtain a stability theorem in the mixed case, one has to work with a given
family of mixed polynomials.

\begin{example}\label{e:B_twisted}
 Let us consider a family of \textit{twisted Brieskorn mixed
polynomials\footnote{terminology used by Oka \cite{Oka1}}}:\\ $F_{s}(\mathbf{z},\overline{\mathbf{z}})=\sum_{i=1}^{n}z_{i}^{a_{i}+b_{i}}\overline{z}_{i}^{b_{i}}+s\sum_{i=1}^{n}z_i^{a_{i}+2b_{i}}$,
where $a_{i},b_{i}\in\mathbb{N}^{+}$ for $1\leq i\leq n$ and $0\leq s<\underset{1\leq i\leq n}{\min}\frac{a_{i}}{a_{i}+2b_{i}}$.
 It turns out by an easy computation that $F_{s}(\mathbf{z},\overline{\mathbf{z}})$
is a family of Newton strongly non-degenerate polynomials. Thus, by our Theorem \ref{t:monodro_family}, the
monodromy at infinity of $F_s$ is isotopic to that of  $F_{0}(\mathbf{z},\overline{\mathbf{z}})= \sum_{i=1}^{n}z_{i}^{a_{i} + b_{i}}\overline{z}_{i}^{b_{i}}$.
\end{example}

\smallskip

For the proof of the theorem, we need some preliminaries. Let $F_{s}$ stand for $F(\mathbf{z},\overline{\mathbf{z}},s)$, let
$F(\Sing F) :=\underset{s\in\left[0,1\right]}{\cup}F_{s}(\Sing F_{s})$, $S(F) :=\underset{s\in\left[0,1\right]}{\cup}S(F_{s})$. We also consider the restriction $F_{s,\Delta}$ of $F_s$ to some face $\Delta$ of $\overline{\supp F_s}$
and write $F_{\Delta}(\mathbf{z},\overline{\mathbf{z}},s) :=F_{s,\Delta}$.

\begin{proposition}\label{p:family_bounded}
Under the assumption of Theorem \ref{t:monodro_family}, the set
$F(\Sing F)\cup S(F)$ is bounded.
\end{proposition}
\begin{proof}
If $F(\Sing F)$ were not  bounded then, by the curve selection Lemma \ref{l:curve-selection}, there exist analytic
paths $\mathbf{z}(t)\in \mathbb{C}^{n}$, $\lambda(t)\in S^{1}$
and $s(t)\in \left[0,1\right]$ defined on a small enough interval
$\left]0,\varepsilon\right[$ such that
\begin{align}
\underset{t\rightarrow0}{\lim}\,\Vert\mathbf{z}(t)\Vert & =\infty,\,\underset{t\rightarrow0}{\lim}\, F(\mathbf{z}(t),\overline{\mathbf{z}}(t),s(t))=\infty,\label{eq:-17}\\
\underset{t\rightarrow0}{\lim}\, s(t)= & s_{0},\,\overline{\mathrm{d}F_{s(t)}}(\mathbf{z}(t),\overline{\mathbf{z}}(t))=\lambda(t)\overline{\mathrm{d}}F_{s(t)}(\mathbf{z}(t),\overline{\mathbf{z}}(t)).\label{eq:-18}
\end{align}

We may then apply the proof of Theorem \ref{t:atyp_Newton}(b) and find a face $\Delta$ of $\overline{\supp(F_{s(t)}^{I})}$, which by assumption is independent of $s$, such that $F_{s(t), \Delta}^{I}$ has a singularity in $\bC^{*n}$. By using Remark \ref{r:f_I}, this contradicts the Newton strong non-degeneracy.

To show that $S(F)$ is bounded, we proceed as follows. By Theorem \ref{t:atyp_Newton}(a), one has the inclusion
$S(F)\subset\cup_{s\in [0,1]}\{F_{s}(0)\}\cup_{s\in [0,1]}\underset{\Delta\in{\mathfrak{B}}_{s}}{\cup}F_{s,\Delta}(\Sing F_{s,\Delta}\cap\mathbb{C}^{*n})$
where ${\mathfrak{B}}_{s}$ is the set of bad faces of $\overline{\supp(F_{s})}$
for $s\in\left[0,1\right]$.
We have that $\cup_{s}\{F_{s}(0)\}$ is bounded by the continuity with respect to $s$, and that 
$\{{\mathfrak{B}}_{s}\}_{s\in [0,1]}$ is a finite set since  $\Gamma^{+}(F_{s})$
is independent of $s$.
If $S(F)$ were not bounded, then we may assume that $F_{\Delta_0(s)}(\Sing F_{\Delta_0(s)}\cap\mathbb{C}^{*n})$
is not bounded as $s\to s_0$, for some  bad face $\Delta_0(s)$ which is actually
independent of $s$ in some small enough interval $\left]s_{0}-\varepsilon,s_{0}+\varepsilon\right[$.
 Since $\Gamma^{+}(F_{s})$ is independent of $s$
and since $\Gamma^{+}(F_{s,\Delta_{0}})\subset\Gamma^{+}(F_{s})$, we
get that $\Gamma^{+}(F_{s,\Delta_{0}})$ is independent of $s$ within a neighbourhood of $s_0$. 
We may then apply the above proof to $F_{\Delta_{0}}$ in place of $F$.
\end{proof}

\begin{proposition}\label{p:r_0}
 Under the  assumption of Theorem
\ref{t:monodro_family}, there exists $r_0>0$ such that, for any $r\ge r_0$,  there exists $R_0(r) \gg 1$ such that one has the transversality
$f_{s}^{-1}(c)\pitchfork S_{R}^{2n-1}$,  $\forall c\in S_{r}^{1}$, $\forall R\geq R_{0}(r)$ and $\forall s\in\left[0,1\right]$.
\end{proposition}
\begin{proof}
The above Proposition \ref{p:family_bounded} implies that there exists $r_{0}>0$ independent on $s\in [0,1]$ such that the following inclusion holds:
\begin{equation}\label{eq:bounded_in_disk}
  F(\Sing F)\cup_{s\in [0,1]}\{F_{s}(0)\}\cup_{s\in [0,1]}\underset{\Delta\in{\mathfrak{B}}_{s}}{\cup}F_{s,\Delta}(\Sing F_{s,\Delta}\cap\mathbb{C}^{*n})\subset\overset{\circ}{D}_{r_{0}}.
\end{equation}
If the above assertion were not true, then by Lemma \ref{l:curve-selection} there
exist analytic paths $\mathbf{z}(t)\subset\mathbb{C}^{n}$, $\lambda(t)\subset\mathbb{R}$, $\mu(t)\subset\mathbb{C}^{*}$ and $s(t)\subset\left[0,1\right]$
such that:
\begin{align}
\underset{t\rightarrow0}{\lim\,}\Vert\mathbf{z}(t)\Vert & =\infty,\,\underset{t\rightarrow0}{\lim\,}F(\mathbf{z}(t),\overline{\mathbf{z}}(t),s(t))=c\in S_{r}^{1}.\label{eq:-21}\\
\underset{t\rightarrow0}{\lim\,}s(t)= & s_{0},\,\lambda(t)\mathbf{z}(t)=\mu(t)\overline{\mathrm{d}F}(\mathbf{z}(t),\overline{\mathbf{z}}(t),s(t))+\overline{\mu}(t)\overline{\d}F(\mathbf{z}(t),\overline{\mathbf{z}}(t),s(t)).\label{eq:-22}\end{align}

By a similar analysis as in the proof of  Theorem \ref{t:atyp_Newton}(a) one finds a singular point $\mathbf{A}\in \bC^{*n}$ of $F_{\Delta}$ where $\Delta$ is either a
face of $\Gamma^{+}(F_{s})$ or a bad face of $\overline{\supp(F_{s})}$.
This contradicts \eqref{eq:bounded_in_disk}.
\end{proof}

\subsection{Proof of Theorem \ref{t:monodro_family}}

By the above two propositions,
for $r\ge r_0$, the global
monodromy fibration
$F_{s\mid}:F_{s}^{-1}(S_{r}^{1})\rightarrow S_{r}^{1}$
 is diffeomorphic to the fibration
\begin{equation}\label{eq:fib_infty}
F_{s\mid}:F^{-1}(S_{r}^{1})\cap B_{R}\rightarrow S_{r}^{1}
\end{equation}
for all  $R\ge R_0(r)$ and all $s\in\left[0,1\right]$.

Consider the map $\tilde F:\mathbb{C}^{n}\times I\rightarrow\mathbb{C}\times I$, $(\mathbf{z},s)\mapsto(F_{s}(\mathbf{z},\overline{\mathbf{z}}),s)$, where $I := [0, 1]$.

The above proposition show that the restriction $\tilde F_|:  \tilde F^{-1}(S_{r}^{1} \times I)\cap (B_{R}\times I) \to S_{r}^{1}
\times I$ is a proper submersion on the couple of manifolds $(\tilde F^{-1}(S_{r}^{1} \times I)\cap(B_{R}\times I), \tilde F^{-1}(S_{r}^{1} \times I)\cap (\partial B_{R}\times I))$.  Then Ehresmann's theorem tells that
the fibrations (\ref{eq:fib_infty}) are isotopic for varying $s$.

\fin

Theorem \ref{t:monodro_family} appears to be useful in finding the topology of the fibres in  Examples \ref{ex:1} and \ref{e:counterex}. 
As another consequence, one may extend the range of applicability of the stability theorems in \cite[Theorem 17]{NZ2} and \cite[Theorem 1.1]{Ph}, as follows:

\begin{corollary}\label{c:newton_constant}
If $f$ and $g$ are two Newton strongly non-degenerate mixed polynomials,
such that $\Gamma^{+}(f)=\Gamma^{+}(g)$ and that their restrictions to the boundaries at 
infinity $f_{\Gamma^+}$ and $g_{\Gamma^+}$  are both holomorphic (or both anti-holomorphic), then the monodromies at infinity
of $f$ and of $g$ are isotopic.
\end{corollary}
\begin{proof}
In the holomorphic setting, the Newton strongly non-degenerate condition is the same as Newton non-degenerate and is a Zariski open and connected condition. This holds for anti-holomorphic instead of holomorphic. This fact allows us to connect $f$ to $g$ by a family of Newton strongly non-degenerate mixed polynomials. For instance, one may do as follows.
First, one applies \cite[Theorem 1.1]{Ph} to the restrictions $f_{\Gamma^+}$ and $g_{\Gamma^+}$. 
Next, we write $f = f_{\Gamma^+} + h$
and observe that the family of mixed polynomials $F_s := f_{\Gamma^+} + (1-s)h$ satisfies the hypotheses of  our Theorem \ref{t:monodro_family} and connects $f$ to $f_{\Gamma^+}$, hence the monodromy is stable in this family. A  similar construction for $g$ completes the picture and ends our proof. 
\end{proof}



\end{document}